\documentclass{amsart}[12pt]
\usepackage{graphics}
\usepackage{amsmath,amsfonts,amssymb,amsthm,mathtools}

\numberwithin{equation}{section}

\usepackage[usesnames,svgnames]{xcolor}
\usepackage{tikz}
\usepackage[enableskew,vcentermath]{youngtab}
\usepackage{ytableau,mathdots,yhmath}
\usepackage[all]{xy}
\usepackage{enumerate}
\usepackage[colorinlistoftodos]{todonotes}
\usepackage{url}
\usepackage{hyperref}

\RequirePackage{cleveref}
\usepackage{hypcap}
\hypersetup{colorlinks=true, citecolor=darkblue, linkcolor=darkblue}
\definecolor{darkblue}{rgb}{0.0,0,0.7}
\newcommand{\darkblue}{\color{darkblue}}
\newcommand{\defn}[1]{\emph{\darkblue #1}}

\newtheorem{theorem}{Theorem}[section]

\newtheorem{proposition}[theorem]{Proposition}

\newtheorem{lemma}[theorem]{Lemma}
\newtheorem{corollary}[theorem]{Corollary}

\newtheorem{example}[theorem]{Example}
\newtheorem{remark}[theorem]{Remark}

\newtheorem*{proposition*}{Proposition}

\DeclareMathOperator{\Red}{\mathcal{R}}
\newcommand{\staircase}{\delta}

\newlength{\cellsize} 
\newsavebox{\cell}
\sbox{\cell}{\begin{picture}(18,18)
\put(0,0){\line(1,0){18}}
\put(0,0){\line(0,1){18}}
\put(18,0){\line(0,1){18}}
\put(0,18){\line(1,0){18}}
\end{picture}}
\newcommand\cellify[1]{\def\thearg{#1}\def\nothing{}%
\ifx\thearg\nothing
\vrule width0pt height\cellsize depth0pt\else
\hbox to 0pt{\usebox{\cell} \hss}\fi%
\vbox to \cellsize{
\vss
\hbox to \cellsize{\hss$#1$\hss}
\vss}}
\newcommand\tableau[1]{\vtop{\let\\\cr
\baselineskip -16000pt \lineskiplimit 16000pt \lineskip 0pt
\ialign{&\cellify{##}\cr#1\crcr}}}

\newcommand{\pic}{\begin{tikzpicture}}
\newcommand{\epic}{\end{tikzpicture}}

\newcommand{\maj}{{{\rm maj}}}

\newcommand{\N}{\mathbb{N}}

\newcommand{\Des}{\mathrm{Des}}
\newcommand{\SSYT}{{{\rm SSYT}}}
\newcommand{\SYT}{{{\rm SYT}}}
\newcommand{\LR}{{{\rm LR}}}
\newcommand{\INC}{\mathrm{INC}}
\newcommand{\Dsh}{D^{\triangledown}}

\newcommand{\RSK}{\mathrm{RSK}}

\newcommand{\ShSYT}{{{\rm ShSYT}}}
\newcommand{\ShSSYT}{{{\rm ShSSYT}}}

\newcommand{\mr}{\mathcal{R}}

\newcommand{\mh}{\mathcal{H}}
\newcommand{\rect}{\mathrm{rect}}

\newcommand{\bfa}{\mathbf{a}}
\newcommand{\bfb}{\mathbf{b}}
\newcommand{\bfi}{\mathbf{i}}
\newcommand{\bfr}{\mathbf{r}}
\newcommand{\FC}{\Phi}

\def\SW{\text{{\rm SW}}}
\def\MS{\text{{\rm MS}}}

\newcommand{\PSW}{{P_{\ts\SW}}}
\newcommand{\QSW}{{Q_{\SW}}}
\newcommand{\PMS}{{P_{\ts\MS}}}
\newcommand{\QMS}{{Q_{\MS}}}

\def\.{\hskip.06cm}
\def\ts{\hskip.03cm}

\def\nin{\noindent}
\def\wt{\widetilde}

\def\emp{\nothing}
\def\sq{\square}

\def\nn{\mathbb N}

\def\la{\lambda}

\def\si{\sigma}
\def\de{\delta}

\def\om{\omega}

\def\vr{\varrho}
\def\vp{\varphi}

\def\wt{\widetilde}
\def\<{\langle}
\def\>{\rangle}

\def\rF{ {\text {\rm F} } }
\def\rG{ {\text {\rm G} } }

\def\0{{\mathbf 0}}

\def\nothing{\varnothing}

\def\.{\hskip.06cm}
\def\ts{\hskip.03cm}

\def\hatequiv{{\ts{}\doteq\.}}
\def\spds{{{}-{}\ts}}

\def\ns{{}}

\newcommand{\Set}{\overline{\SYT}}

\newcommand{\ShSet}{\overline{\ShSYT}}
\newcommand{\ShSSet}{\overline{\ShSSYT}}

\newcommand{\PSH}{P_{SH}}
\newcommand{\QSH}{Q_{SH}}
\newcommand{\res}{\text{\sf res}}

\title[Bijecting hidden symmetries]{Bijecting hidden symmetries for skew staircase shapes}

\author[Hamaker]{Zachary Hamaker}
\address{Department of Mathematics,  University of Florida, Gainesville, FL 32611}
\email{zhamaker@ufl.edu}

\author[Morales]{Alejandro H.\ Morales}
\address{Department of Mathematics and Statistics, University of Massachusetts, Amherst, MA 01003}
\email{amorales@math.umass.edu}
	
\author[Pak]{Igor Pak}
\address{Department of Mathematics, University of California, Los
		Angeles, CA 90095}
\email{pak@math.ucla.edu}

\author[Serrano]{Luis Serrano}
\address{Zapata Computing Canada Inc., 325 Front St.~W, Toronto, ON, M5V 2Y1}
\email{luisgui.serrano@gmail.com}
	
\author[Williams (cinco amigos)]{Nathan Williams}
\address{Department of Mathematical Sciences, University of Texas at Dallas,
Richardson, TX 75080}
\email{nathan.williams1@utdallas.edu}

\thanks{\today}

\begin{document}

\begin{abstract}
We present a bijection between the set \ts $\SYT(\la/\mu)$ \ts of
standard Young tableaux of \emph{staircase minus rectangle shape} $\la=\de_k$, $\mu=(b^a)$,
and the set \ts $\ShSYT'(\eta)$ \ts of \emph{marked shifted standard Young tableaux} of a
certain shifted shape $\eta=\eta(k,a,b)$.  Numerically, this result is due to DeWitt~(2012).
Combined with other known bijections this gives a bijective proof of the product
formula for $|\SYT(\la/\mu)|$.  This resolves an open problem by Morales, Pak and
Panova (2019), and allows an efficient random sampling from $\SYT(\la/\mu)$.
Other applications include a bijection for semistandard Young tableaux, and a
bijective proof of Stembridge's symmetry of LR--coefficients of the staircase
shape.  We also extend these results to \emph{set-valued standard Young tableaux}
in the \emph{combinatorics of $K$-theory}, leading to new proofs of results by Lewis and Marberg (2019) and  Abney-McPeek,
An and Ng~(2020).
\end{abstract}

\maketitle

\section{Introduction}

The phrase `hidden symmetries' in the title refers to coincidences
between the numbers of seemingly different (yet similar) sets of combinatorial objects.
When such coincidences are discovered, they tend to be
fascinating because they reflect underlying algebraic
symmetries --- even when the combinatorial objects themselves appear to possess
no such symmetries.

It is always a relief to find a simple combinatorial explanation
of hidden symmetries.  A direct bijection is
the most natural approach, even if sometimes such a bijection is both hard to find
and to prove (cf.~$\S$\ref{ss:finrem-LR}).  Such a bijection restores order
 to a small corner of an otherwise disordered universe,
suggesting we are on the right path in our understanding.
It is also an opportunity to learn more about our combinatorial objects.

\smallskip

\subsection*{The results}
We start with the following unusual product formula.  Denote by \ts
$\de_k=(k-1,\ldots,2,1)$ \ts the \defn{staircase shape}.

\medskip

\begin{theorem}[Staircase minus rectangle, see below] \label{c:dewitt}
For all \ts  $a,b,c\in \nn$, let $\lambda=\de_{a+b+2c}$, and $\mu=(b^a)$.
Then the number  \ts $f^{\lambda/\mu} = \bigl|\SYT(\lambda/\mu)\bigr|$ \ts  is equal to
\[
n! \,  \,  \frac{\rF(a)\.\rF(b)\.\rF(c)\.
\rF(a+b+c)\ts\cdot\ts \rG(c)\.\rG(a+b+c)}{\rF(a+b)\.\rF(b+c)
\.\rF(a+c)\ts\cdot\ts  \rG(a+c)\.\rG(b+c)\.\rG(a+b+2c)}\,,
\]
where  \ts $n=|\la/\mu|=\binom{a+b+2c}{2}-ab$, \ts $\rF(m) :=1!\cdot 2! \ts \cdots \ts (m-1)!$,
and \ts $\rG(m):= \. 1!! \cdot 3!! \ts \cdots \ts (2m-3)!!$
\end{theorem}

\smallskip

This curious formula was first derived by DeWitt~\cite{DeW} in a somewhat different
form.  The $q$-version was given in~\cite{KS}, and further generalizations were obtained
in~\cite{MPP3} for a more general class of skew shapes.
%
To understand the product formula in the theorem, consider the following:

\smallskip

\begin{theorem}[DeWitt]\label{t:DeWitt-SYT}
Let $\la=\de_k$ be a staircase,
$\mu=(b^a)$ be a rectangle, such that \. $a+b<k$.   Then:
\begin{equation}\label{e:main}
\bigl|\SYT(\la/\mu)\bigr| \, = \, 2^N \. \bigl|\ShSYT(\eta)\bigr|\ts,
\end{equation}
where \ts $\ShSYT(\eta)$ \ts is the set of \emph{shifted standard Young tableaux}
of shifted shape \ts $\eta=\eta(k,a,b)$ \ts defined in \ts {\rm ~\Cref{f:intro}}, and \ts
$N:=|\eta|-k+a = \binom{k}{2}-ab-k+a$.
\end{theorem}

\smallskip

\begin{figure}[hbt]
		\centering
		\includegraphics[height=3.2cm]{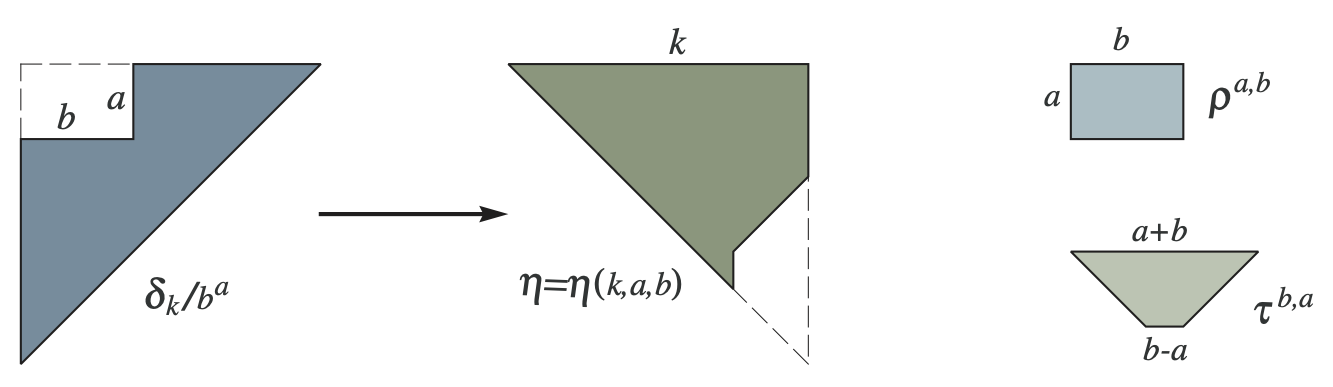}
		\caption{Main bijection \. $\vp \ts : \. \SYT(\de_k/b^a) \ts\longrightarrow \ts\ShSYT'(\eta)$.}
		\label{f:intro}
	\end{figure}

\smallskip

Our main result is a bijective proof of~\eqref{e:main}, where we interpret
the RHS as the number \ts $\bigl|\ShSYT'(\eta)\bigr|$\ts of certain
\emph{marked shifted standard Young tableaux} (see below).  For the case of straight
shapes, i.e.\ for $a=b=0$ and $\mu=\emp$, an equivalent bijection was given
by Purbhoo~\cite{Pur} (see $\S$\ref{ss:finrem-marv}).
Our bijection has additional properties, as it extends to the proof of a
symmetric function identity (see~\Cref{t:dewitt}).  \Cref{c:dewitt}
follows from~\Cref{t:DeWitt-SYT} and the \emph{hook-length formula} (HLF)
for shifted shapes due to Thrall~\cite{Thr}.  Combined with the bijection
in~\cite{Fis}, this gives the first direct bijective proof of~\Cref{c:dewitt},
resolving an open problem in~\cite[$\S$9.2]{MPP3}.

\smallskip

\begin{theorem}\label{t:DeWitt-random}
Let $\la=\de_k$ be a staircase, $\mu=(b^a)$ be a rectangle, s.t.\ $a+b<k$.
Then there is a $O(k^3\log k)$ time algorithm for uniform random generation of
standard Young tableaux in $\SYT(\la/\mu)$.
\end{theorem}

\smallskip

For the proof, we combine the bijection for~\eqref{e:main} and the known
uniform random generation algorithms for shifted shapes: either the
NPS--style \emph{two-dimensional bubble sorting}\footnote{This type of insertion is a close
relative to \emph{jeu-de-taquin} and \emph{promotion}, all heavily studied in this context.}
in~\cite{Fis}, or the GNW--style
\emph{hook walk} in~\cite{Sag-GNW}.

The uniform random generation of combinatorial objects is a classical problem
that is well understood for many planar structures (see~$\S$\ref{ss:finrem-random}).
For standard Young tableaux
of skew shapes, the iterative application of \emph{Feit's determinantal formula}~\cite{Feit}
gives an easy \ts $O\bigl(k^{2+\om}\bigr)$ \ts time algorithm.
Here \ts $\om\ge 2$ \ts is the \ts \emph{matrix multiplication constant} \ts with currently
best known upper bound \ts $\om<2.3729$~\cite{AW}. Our algorithm is thus a substantial
improvement over this approach.

\smallskip

Our final application is a bijective proof of the following unusual
symmetry of \emph{Littlewood--Richardson} (LR--) \emph{coefficients} \.
$c^\la_{\mu\nu}=\bigl|\LR(\la/\mu,\nu)\bigr|$, where
\ts $\LR(\la/\mu,\nu)\subseteq \SSYT(\la/\mu,\nu)$ \ts denotes the set of
\emph{LR--tableaux} (see e.g.~\cite[$\S$4.9]{Sag-book}).
\smallskip

\begin{corollary}[{of~\Cref{t:skew}}]\label{c:skew-LR}
Let $\mu,\nu \subseteq \staircase_k$, such that $|\mu|+|\nu|=|\staircase_k|$.
Then: \ts
\begin{equation}\label{e:Stem-LR}
c^{\delta_k}_{\mu\nu} \, = \, c^{\delta_k}_{\mu'\nu} \ ,
\end{equation}
where $\mu'$ denotes the \emph{conjugate partition} of~$\mu$.
\end{corollary}

\smallskip

Using other symmetries of the LR--coefficients, equation~\eqref{e:Stem-LR}
implies
\begin{equation}\label{eq:LR-sym}
c^{\delta_k}_{\mu\nu} \, = \, c^{\delta_k}_{\mu'\nu} \, = \,c^{\delta_k}_{\mu\nu'}\, = \, c^{\delta_k}_{\mu'\nu'}
\, = \, c^{\delta_k}_{\nu'\mu'} \, = \, c^{\delta_k}_{\nu\mu'} \, = \, c^{\delta_k}_{\nu'\mu} \, = \,c^{\delta_k}_{\nu\mu} \ .
\end{equation}
Note that there are two more symmetries (see e.g.~\cite{BZ}):
$$
c^{\delta_k}_{\mu\nu} \, = \, c^{\mu^\ast}_{\delta_k\nu} \, = \, c^{\nu^\ast}_{\mu \delta_k}\,,
$$
where \ts $\lambda^\ast$ \ts denotes the complement to~$\lambda$ in the \ts
$(k-1)\times k$ \ts rectangle, so that \ts $\delta_k^\ast = \delta_k$.
This triples the number of equal LR--coefficients given by~\eqref{eq:LR-sym},
with the staircase $\de_k$ as one of the partitions.
In the next section we explain the algebra behind both~\Cref{t:DeWitt-SYT}
and~\Cref{c:skew-LR}.

\medskip

\subsection*{Algebraic interpretation}
The ring of symmetric functions $\Lambda$ has many bases indexed by integer partitions, including the Schur functions $s_\lambda$.
The subring $\wt{\Lambda}$ generated by odd power sums $p_k = x_1^k + x_2^k + \dots$ has bases indexed by strict integer partitions,
one being the \defn{Schur $P$--functions}.
Our proof of~\Cref{t:DeWitt-SYT} leads to bijective explanations for some relations between skew staircase Schur functions and Schur $P$--functions.

To state the identities in question, we introduce some notation.
Let \ts $\mu = (\mu_1, \ldots,\mu_k)$ \ts and \ts
$\lambda = (\lambda_1, \ldots, \lambda_k)$ \ts be integer partitions, and let $\mu'$ be the transpose of~$\mu$.
We write $\mu \subseteq \lambda$ if $\mu_i \leq \lambda_i$ for all~$i$, and let $s_{\lambda/\mu}$
denote the associated \defn{skew Schur function}.

A partition $\lambda$ with \ts $\lambda_1 > \cdots > \lambda_k$ \ts is called \defn{strict}.
When $\mu \subseteq \lambda$ are both strict, let $P_{\lambda/\mu}$ denote the associated \defn{skew Schur $P$--function}.
As before, let $\staircase_n = (n-1,n-2,\dots,1)$ be the staircase partition and $\vr_n$ be the same partition, viewed as a shifted shape.

\smallskip

\begin{theorem}[{J.~Stembridge, 2004, see~$\S$\ref{ss:finrem-hist}}]\label{t:skew}
	Let \ts $\mu \subseteq \staircase_n$\ts.
	Then \. $s_{\delta_n/\mu} = s_{\delta_n/\mu'}$\ts.
\end{theorem}

\smallskip

Stembridge's theorem immediately implies~\Cref{c:skew-LR};
it is in fact equivalent to the corollary.  The next result is an algebraic
explanation of the phenomenon in Stembridge's theorem:

\smallskip

\begin{theorem}[{\cite[Thm~4.10]{AS}}]\label{t:skew-AS}
	Let $\mu \subseteq \staircase_n$.  Then \ts $s_{\delta_n/\mu} \in \wt \Lambda$.
	Moreover, in the expansion \ts $s_{\delta_n/\mu} \. = \. \sum_{\nu} \ts  a_\nu P_\nu$\ts, \ts
all $a_\nu$ \ts are non-negative integers.
\end{theorem}
%
%
This theorem was originally conjectured by Stanley in 2001,
and proved by Ardila and Serrano using algebraic methods.  We give
a bijective proof of~\Cref{t:skew} that also implies~\Cref{t:skew-AS}, see bijection in~\eqref{e:stem} below.

\medskip

For $\mu \subseteq \vr_n$ strict, write $\vr_n\spds\mu$ for the partition whose diagram is obtained by reflecting the diagram of $\vr_n/\mu$ across the line $y=x$.
Let $\rho^{\ell,m} = (m^\ell)$ be the $\ell \times m$ rectangle partition, and when $\ell > m$ let $\tau^{\ell,m} = (\ell{+}m{-}1,\ell{+}m{-}3,\dots,\ell{-}m{+}1)$ be the shifted trapezoid, see~\Cref{f:intro}.

\smallskip

\begin{theorem}[{\cite[Thm~V.3]{DeW}}] \label{t:dewitt}
Let \ts $\ell + m < n$. Then:
$$
s_{\delta_n/\rho^{\ell,m}} \. = \. P_{\vr_n/\tau^{\ell,m}} \. = \. P_{\vr_n\spds\tau^{\ell,m}}
$$
\end{theorem}

\smallskip

Recall that Schur functions are weighted generating functions of semistandard Young tableaux (see e.g.~\cite{Mac,EC}).
Similarly, Schur $P$--functions are weighted generating functions of semistandard Young tableaux of shifted shape with marked entries~\cite{Ste1}.
Semistandard tableaux can be associated to standard tableaux via Gessel's \emph{fundamental quasisymmetric functions} using bijective methods.

To prove~\Cref{t:dewitt}, we construct a bijection between standard tableaux and shifted standard tableaux with marked entries that reverses descent sets.
We do so using \emph{Worley--Sagan insertion}~\cite{Sag,Wor}, and the combinatorics of reduced words for \emph{fully commutative permutations}~\cite{BJS,Ste2}.\footnote{These are also known as 321-\emph{avoiding permutations}~\cite{BJS}.
We will not use this characterization.}
Our construction is easily seen to be surjective; we prove injectivity via a related map constructed using RSK and mixed shifted insertion.

In~$\S$\ref{ss:finrem-Schubert}, we recall an elegant geometric interpretation of~\Cref{t:dewitt}.
This geometric interpretation suggests that \Cref{t:skew,t:skew-AS,t:dewitt} should extend to
\emph{stable Grothendieck polynomials} and $K$-theoretic analogues of Schur $P$-functions.
We prove these extensions in \Cref{s:K} as \Cref{t:Groth-skew,t:K-dewitt} and \Cref{c:AMAN}, the first appearing in~\cite{LM} and the last being a main result in~\cite{AMAN}.
Our proofs of \Cref{t:Groth-skew} and \Cref{c:AMAN} are combinatorial as opposed to the previous algebraic proofs, but our proof of \Cref{t:K-dewitt} uses algebraic identities.
As a consequence, in \Cref{t:set-dewitt} we extend our bijective proof of \Cref{t:DeWitt-SYT} to a bijection between
certain \emph{set-valued tableaux}.

\medskip

\subsection*{Paper structure} In the lengthy and detailed~\Cref{s:3bij}, we recall
much of the background on the usual and shifted Young tableaux, reduced words and insertion algorithms.
Although our notation are self-contained, our arguments are technical and rely on many pieces of Young tableau technology.
This appears to be unavoidable for a self-contained proof that our map is bijective, and we disperse references to the literature
throughout the paper.
A reader willing to assume \Cref{t:dewitt} can find a concise description of our bijection at the beginning of \Cref{s:main}, with proof in \Cref{p:inverse} (see also the equivalent maps discussed after \Cref{l:first-part} and in~$\S$\ref{ss:finrem-crystal}).

We prove our main results in~\Cref{s:main}.   Our arguments are both dense and concise, so
examples are added for clarity.  In a short \Cref{s:K},
we give a generalization of both DeWitt's and Stembridge's theorems to \emph{Grothendieck polynomials}
and extend our bijections to work in this case.  We conclude with final remarks in~\Cref{s:finrem},
including some large simulations.

\bigskip

\section{Tableaux, reduced words and insertion algorithms}\label{s:3bij}
\label{s:defns}
\smallskip

\subsection{Basic notation}\label{ss:3bij-basic}
We write $\nn=\{0,1,\ldots\}$, $[n]=\{1,\ldots,n\}$ and $2^S$ for the power set of $S$. We also fix the
linear ordering \. $1'<1<2'<2<\.\cdots\.<n'<n$ \. of \defn{marked integers}.

\smallskip

\subsection{Young diagrams and Young tableaux}\label{ss:3bij-def}
Recall an integer partition is a sequence \ts
$\lambda = (\lambda_1, \dots, \lambda_k)$ of non-negative integers \ts $\la_1\ge \ldots \ge \la_k \ge 0$.
We identify $\lambda$ with its \defn{Young diagram}
\[
D_\lambda = \{(i,j) \in \N^2: 1 \leq j \leq \lambda_i\},
\]
which we view as a poset via pointwise comparison, so $(i,j) \leq (k,\ell)$ if $i \leq k$ and $j \leq \ell$.
Given $\mu \subseteq \lambda$, i.e., $\mu$, $\lambda$ such that $D_\mu \subseteq D_\lambda$, the \emph{skew partition} $\lambda / \mu$ corresponds to the diagram $D_\lambda \setminus D_\mu$.
A \defn{standard Young tableau} of shape $\lambda/\mu$ is a linear extension of $D_\lambda\setminus D_\mu$, the set  of which is denoted $\SYT(\lambda/\mu)$.

There is a parallel theory for \defn{strict partitions} $\lambda$, which satisfy $\lambda = (\lambda_1 > \dots > \lambda_k)$.
The \defn{shifted Young diagram} of a strict partition $\lambda$ is
\[
\Dsh_\lambda = \{(i,j) \in \mathbb{N}^2: i \leq j \leq \lambda_i + i - 1\},
\]
which we again view as a poset via pointwise comparison.
The \defn{diagonal} of $D_\lambda$ is the subset $\{(i,i):i \in \N\} \cap D_\lambda$.
For $\mu,\lambda$ strict partitions with $\mu \subseteq \lambda$, a \defn{shifted standard Young tableau} of shifted shape $\lambda/\mu$ is a linear extension of $\Dsh_{\lambda} \setminus \Dsh_\mu$.
Let $\ShSYT(\lambda/\mu)$ denote the set of shifted standard Young tableaux of shape $\lambda/\mu$.
A \defn{shifted standard Young tableau with marked entries} of shape $\lambda/\mu$ is a pair $(T,S)$ where $T \in \ShSYT(\lambda/\mu)$ and $S \subseteq D_{\lambda} \setminus D_\mu$ such that $S$ contains no diagonal entries.
Let $\ShSYT'(\lambda/\mu)$ be the set of shifted standard Young tableaux with marked entries of shape $\lambda/\mu$.
The entries in $S$ are \defn{marked} with symbol~$'$, see~\Cref{fig:SYT}.
For $T$ a tableau, let $T_{ij} = T((i,j))$ denote its $(i,j)$th entry.

Note for any skew partition $\lambda/\mu$ that $D_{\lambda/\mu}$ and $\Dsh_{\lambda+\delta_n/\mu+\delta_n}$ are isomorphic as posets, so the theory of (unshifted) Young tableaux can be viewed as a special case of theory of shifted Young tableaux.
Lastly, we recall a \defn{semistandard Young tableau} of shape $\lambda/\mu$ is tableau with positive integer entries whose rows are weakly increasing and columns are strictly increasing.
Let $\SSYT(\lambda/\mu)$ denote set of semistandard Young tableaux of shape $\lambda/\mu$.
For shifted shapes, instead use $\ShSSYT(\lambda/\mu)$.
The set $\ShSSYT'(\lambda/\mu)$ of \defn{marked shifted semistandard Young tableaux} allows the off-diagonal entries of a shifted semistandard Young tableaux to be marked by the symbol~$'$, with at most one $i$ in each column and at most one $i'$ in each row. 

\begin{figure}[hb]
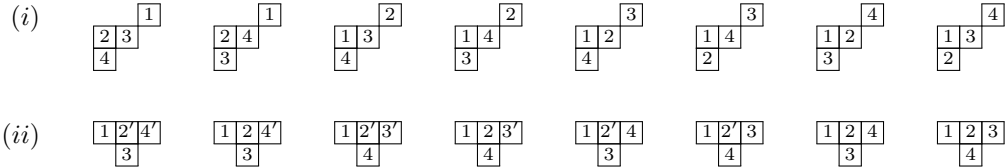

\begin{align*}\qquad (i) & \qquad
\ytableausetup{smalltableaux}
\begin{ytableau}
	\none & \none & 1\\
	2 & 3\\
	4
\end{ytableau} \qquad
\begin{ytableau}
	\none & \none & 1\\
	2 & 4\\
	3
\end{ytableau} \qquad
\begin{ytableau}
	\none & \none & 2\\
	1 & 3\\
	4
\end{ytableau} \qquad
\begin{ytableau}
	\none & \none & 2\\
	1 & 4\\
	3
\end{ytableau} \qquad
\begin{ytableau}
	\none & \none & 3\\
	1 & 2\\
	4
\end{ytableau} \qquad
\begin{ytableau}
	\none & \none & 3\\
	1 & 4\\
	2
\end{ytableau} \qquad
\begin{ytableau}
	\none & \none & 4\\
	1 & 2\\
	3
\end{ytableau} \qquad
\begin{ytableau}
	\none & \none & 4\\
	1 & 3\\
	2
\end{ytableau} \\
& \ & \ & \ & \ & \ & \ \ & \ & \\
\qquad (ii) & \qquad
\begin{ytableau}
	1 & 2' & 4'\\
	\none & 3
\end{ytableau} \qquad
\begin{ytableau}
	1 & 2 & 4'\\
	\none & 3
\end{ytableau} \qquad
\begin{ytableau}
	1 & 2' & 3'\\
	\none & 4
	\end{ytableau} \qquad
\begin{ytableau}
	1 & 2 & 3'\\
	\none & 4
\end{ytableau} \qquad
\begin{ytableau}
	1 & 2' & 4\\
	\none & 3
\end{ytableau} \qquad
\begin{ytableau}
	1 & 2' & 3\\
	\none & 4
\end{ytableau} \qquad
\begin{ytableau}
	1 & 2 & 4\\
	\none & 3
\end{ytableau} \qquad
\begin{ytableau}
	1 & 2 & 3\\
	\none & 4
\end{ytableau}
\end{align*}
\caption{The sets \. $(i)$ \ts $\SYT(\staircase_4/(2))$ \. and \. $(ii)$ \ts $\ShSYT'(\vr_4 - (2))$, with columns corresponding to our bijection.}
\label{fig:SYT}
\end{figure}

The reading word of an unmarked tableau $T$ is the word $r(T)$ obtained by reading its rows from left to right, bottom to top.
For example, the reading word of the third tableau in~\Cref{fig:SYT}$(i)$ is $4132$.

\subsection{Descents and symmetric functions}
\label{ss:des}
For $T$ a tableau of shape $\lambda$, let \. $x^T = \prod_{(i,j) \in D_\lambda} x_{T_{ij}}$,
where \ts $T_{ij}$ \ts is the $(i,j)$-th entry in~$T$, and \ts $x_{i'} = x_i$ \ts for marked entries.
The \defn{Schur functions} and \defn{Schur $P$-functions} are defined as
\begin{equation}
\label{eq:schur}
s_\lambda = \sum_{T \in \SSYT(\lambda)} x^T \quad \mathrm{and} \quad P_\mu = \sum_{T \in \ShSSYT'(\mu)} x^T \quad (\mu\ \mathrm{strict}).
\end{equation}
As we will see, there is a natural way to partition semistandard tableaux into sets indexed by standard tableaux.
This will give a formula for $s_\lambda$ and $P_\mu$.

For $\bfa = (a_1,\ldots, a_p)$ a word in $\mathbb{Z}$, the \defn{descent set} of $\bfa$ is $\Des(\bfa) = \{i \in [p-1]: a_i > a_{i+1}\}$.
Similarly, for $T \in \SYT(\lambda)$ with $\lambda$ a partition of $n$, the \defn{descent set} of $T$ is defined as
\[
\Des(T) \, := \, \bigl\{i \in [n-1]~:~i \ \, \text{is strictly above} \ \, i+1\bigr\}.
\]
Here by \defn{strictly above} we mean that the row number with entry~$i$ is strictly smaller than
that with~$(i+1)$.
For $T \in \ShSYT'(\mu)$ with $\mu$ a strict partition of size $n$, the \emph{descent set} of $T$ is defined as
\[
\Des(T) \, := \, \bigl\{i \in [n-1]: i,(i+1)' \in T\., \, \text{\text{\underline{or}}} \ \, i\ \text{\rm is strictly above}\ i+1\., \ \,
\text{\text{\underline{or}}} \ \, i' \ \text{\rm is strictly to the left of} \ (i+1)'\bigr\}.
\]
These sets are closely related to each other, as we will see in \Cref{t:sw}.

For $T \in \ShSYT'(\lambda)$, let $T^\circ$ be the tableau obtained by marking all unmarked off-diagonal entries and unmarking all marked entries.
For example, if $T$ if the first tableau in~\Cref{fig:SYT}~$(ii)$, then $T^\circ$ is the seventh.
Given $S \subseteq [n-1]$, its \defn{complement} is $S^c = [n-1] \setminus S$.
The next result follows from the definition of $\Des(T)$.

\begin{lemma}
\label{l:complement}
	For $\lambda$ a shifted shape with $|\lambda| = n$ and $T \in \ShSYT'(\lambda)$, we have:
	\[
	\Des(T^\circ) \. = \. \Des(T)^c.
	\]
\end{lemma}

Let $S \subseteq [n-1]$, and define $I_S = \{(i_1 \leq \dots \leq i_n): k \in S \Rightarrow i_k < i_{k+1}\}$.
For $\bfi = (i_1,\dots,i_n)$ and $T$ a shifted marked standard tableau with $n$ entries, let $T(\bfi)$ be the tableau obtained from $T$ by replacing entries labeled $k$ with $i_k$ and entries labeled $k'$ with $i_k'$ for each $k \in [n]$.

\smallskip

\begin{theorem}
\label{t:fundamental}
Let $\lambda$ be a partition of $n$ and $\mu$ be a strict partition of $n$.
	Then \ts $(T,\bfi) \mapsto T(\bfi)$ is a bijection from
\[
(a)\  \bigcup_{T \in \SYT(\lambda)} \{T\} \times I_{\Des(T)} \to \SSYT(\lambda) \quad \quad (b)\ \bigcup_{T \in \ShSYT'(\mu)} \{T\} \times I_{\Des(T)} \to \ShSSYT'(\mu).
\]
\end{theorem}

\smallskip

Part~(a) of~\Cref{t:fundamental} is essentially equivalent to Gessel's formula
for Schur functions in terms of fundamental quasisymmetric functions, while part~(b)
is equivalent to Stembridge's analogous formula for Schur $P$--functions~\cite{Ste1}, both of which we now state:

\smallskip

\begin{corollary}\label{c:fundamental}
Let $\lambda$ be a partition and $\mu$ be a strict partition.
	\[
	s_{\lambda} = \sum_{T \in \SYT(\lambda)} \sum_{\bfi \in I_{\Des(T)}} x^{T(\bfi)} \quad \mathrm{and} \quad P_{\mu} = \sum_{T \in \ShSYT'(\mu)} \sum_{\bfi \in I_{\Des(T)}} x^{T(\bfi)}.
	\]
\end{corollary}

For $S \subseteq [n-1]$, its \defn{reverse} is $S^r = \{n-k: k\in S\}$.
The following is a corollary of~\cite[A1.2.11]{EC}, and can be proved bijectively using Sch\"utzenberger's evacuation involution.

\smallskip

\begin{lemma}\label{l:rev-comp} Let $\lambda$ be a partition.
	\[s_{\lambda} = \sum_{T \in \SYT(\lambda)} \sum_{\bfi \in I_{\Des(T)^r}} x^{T(\bfi)}.\]
\end{lemma}

As a consequence, we see~\Cref{t:dewitt} follows from the existence of a bijection between $\SYT(\delta_n /\rho_{r,\ell})$ and $\ShSYT'(\vr_n/\tau_{r,\ell})$ that reverses descent sets.

 \subsection{Insertion algorithms and jeu de taquin}
\label{ss:ws}

To construct our bijection, we use RSK as well as three closely related bijections for shifted tableaux: Worley--Sagan insertion, mixed shifted insertion and jeu de taquin.  We assume the reader is familiar with ordinary RSK (see e.g.~\cite{Sag-book,EC}), which we denote \ts
$\bfa \xmapsto{\mathrm{RSK}} (P(\bfa),Q(\bfa))$.

\

\noindent\textbf{Weak and strict insertion:}
For $x$ a letter and $\bfr = (r_1, \dots, r_n)$ a word, let $i$ be the smallest index so that $r_i > x$ or $n+1$ if $r_n \leq x$.
We \defn{weak insert} $x$ into $\bfr$ by replacing $r_i$ with $x$, \defn{bumping} $r_i$ if $i \neq n+1$.
Similarly, if $j$ is the smallest index so that $r_j \geq x$ or $n+1$ if $r_n <x$, we \defn{strict insert} $x$ into $r$ by replacing $r_j$ with $x$, bumping $r_j$ if $j \neq n+1$.

\

\noindent\textbf{Worley--Sagan bumping:}
For $T$ a shifted tableau and $x := x^1$ a letter, insert $x^1$ into $T$ by:
\begin{enumerate}
\item Weak insert $x^i$ into the $i$-th row of $T$ (viewing this row as a word); if no entry is bumped, then the process terminates.
\item If $x^i$ bumps a letter $x^{i+1}$ that isn't the first entry of the $i$-th row of $T$, return to step (1) and continue to weak insert into rows.
\item Otherwise, if $x^i$ bumps the first entry of the $i$-th row of $T$, begin to strict insert $x^{j}$ into the $j$-th column of $T$ for $j = i+1, i+2,\ldots$ (again, viewing the column as a word) until no entry is bumped, then terminate.
\end{enumerate}

\

\noindent \textbf{Worley--Sagan insertion:} Given a word $\bfa = (a_1,\dots, a_p)$, initialize $\PSW$ and $\QSW$ as empty tableaux.
For $i = 1,2,\ldots,p$, insert $a_i$ into $\PSW$ using Worley--Sagan bumping, resulting in a new cell $c$.
In $\QSW$, we label $c$ with $i$ if the final insertion was a row insertion and $i'$ if the final insertion was a column insertion.

\medskip

The \defn{Knuth relations} are the transformations:\footnote{See e.g.~\cite[$\S$A1.1]{EC}, where these are called
\defn{Knuth equivalences} and \defn{Knuth transformations}.}
\[
acb \leftrightarrow cab \quad \mathrm{with}  \quad a \leq b < c \quad  \mathrm{and} \quad bac \leftrightarrow bca \quad \mathrm{with} \quad a < b \leq c.
\]
For a word $\bfa = (a_1, \dots, a_p)$, a \defn{Knuth move} is the application of a Knuth relation to a consecutive triple $a_i a_{i+1} a_{i+2}$, and a \defn{shifted Knuth move} is a Knuth move or the exchange of the first two entries $a_1$ and $a_2$.

We say $\bfa$ and $\bfb$ are \defn{Knuth equivalent}, denoted $\bfa\ts \equiv\ts \bfb$, if  they differ by a sequence of Knuth moves.
Similarly, $\bfa$ and $\bfb$ are \defn{shifted Knuth equivalent}, denoted $\bfa \ts\hatequiv\ts \bfb$, if they differ by a sequence of shifted Knuth moves.
Note that $\bfa \equiv \bfb$ implies $\bfa \hatequiv \bfb$, but the converse need not hold.
The equivalence classes under the relations $\equiv$ and $\hatequiv$ are called \defn{Knuth classes} and \defn{shifted Knuth classes}, respectively.

\smallskip

\begin{example}\label{ex:ws}{\rm
Applying Worley--Sagan insertion to the words $\bfa = (1,3,2,5,4,3)$ and $\bfb = (3,1,5,2,4,3)$, we see

\[
\PSW(\bfa) = \PSW(\bfb) = \raisebox{\height}{\begin{ytableau}
1 & 2 & 3\\
\none & 3 & 4\\
\none & \none & 5
\end{ytableau}}\ ,
\hspace{10pt}
\QSW(\bfa) \ = \ \
\raisebox{\height}{\begin{ytableau}
1 & 2 & 4\\
\none & 3 & 5\\
\none & \none & 6
\end{ytableau}}\ ,
\hspace{10pt}
\QSW(\bfb) \ = \ \ \raisebox{\height}{\begin{ytableau}
1 & 2' & 3\\
\none & 4 & 5\\
\none & \none & 6
\end{ytableau}}.
\]

\smallskip

\nin
Note that \ts $\bfa \equiv \bfb$ \ts (so that \ts $\bfa \hatequiv \bfb$)  via \ts
$\bfa = (\underline{1,3,2},5,4,3) \leftrightarrow (3,1,\underline{2,5,4},3) \leftrightarrow (3,1,5,2,4,3) = \bfb$.
}\end{example}

\smallskip

We summarize some key properties of Worley--Sagan insertion.

\smallskip

\begin{theorem}[{\cite{Sag,Wor}}]
\label{t:sw}\
\begin{enumerate}
\item Worley--Sagan insertion is a bijection from words to pairs $(\PSW,\QSW)$ of tableaux, where $\PSW \in \ShSSYT(\lambda)$ and $\QSW \in \ShSYT'(\lambda)$ for some shifted shape $\lambda$.
\item The words $\bfa$ and $\bfb$ are shifted Knuth equivalent if and only if $\PSW(\bfa) = \PSW(\bfb)$.
\item $\Des(\bfa) = \Des(\QSW(\bfa))$.
\end{enumerate}

\end{theorem}

\smallskip

For \ts $\bfa = (a_1,\dots,a_p)$, denote \ts $-\bfa := (-a_1,\dots,-a_p)$.

\smallskip

\begin{theorem}[{\cite[Cor.~8.9]{Hai1}}]
\label{t:neg}  In the notation above, we have: \ts
$\QSW(-\bfa) = \QSW(\bfa)^\circ$.
\end{theorem}

\smallskip

Next, we introduce \defn{mixed shifted insertion}.

\

\noindent \textbf{Mixed shifted bumping:} For $T$ a shifted marked tableau and $x := x^1$ an unmarked letter, set the index of the initial cell to be $(y_1,z_1) = (1,1)$.

\begin{enumerate}
\item If $x^i$ is unmarked, weak insert $x^i$ into the $(y_i+1)$st row of $T$; if $x^i$ is marked, weak insert $x^i$ into the $(z_i+1)$-st column of $T$; if no entry is bumped, then the process terminates.
\item If $x^i$ bumps a letter, define $(y_{i+1},z_{i+1})$ to be the cell of the bumped entry and $x^{i+1}$ to be the bumped entry if $y_{i+1} < z_{i+1}$ and $(x^{i+1})'$ if $y_{i+1} = z_{i+1}$.  Return to step (1).
\end{enumerate}

In other words, when a letter is bumped from the diagonal (which is unmarked, by definition), the letter becomes marked and gets weak inserted into the next column.

\

\noindent \textbf{Mixed shifted insertion:} Given a word $\bfa = (a_1,\dots, a_p)$, initialize $\PMS $ and $\QMS$ as empty tableaux.
For $i \in [p]$, insert $a_i$ into $\PMS$ using mixed shifted bumping.
This process terminates with a new cell $c$, which is added to $\QMS$ with label $i$.

\smallskip

\begin{theorem}[{\cite{Hai1}}] \ 
\label{t:mixed}
	\begin{enumerate}
		\item Mixed shifted insertion is a bijection from permutations to pairs $(\PMS,\QMS)$ of tableaux, where $\PMS \in \ShSYT'(\lambda)$ and $\QMS \in \SYT(\lambda)$ for some shifted shape $\lambda$.
		\item For $w$ a permutation, $\PMS(w) = \QSW(w^{-1})$ and $\QMS(w) = \PSW(w^{-1})$.
	\end{enumerate}
\end{theorem}

\smallskip


We now define jeu de taquin.
An \defn{inner corner} of the skew partition $\lambda/\mu$ is a maximal element of $D_\mu$ and an \defn{outer corner} is a minimal entry of $\mathbb{N}^2 \setminus D_{\lambda}$.
When $\lambda$ and $\mu$ are strict partitions, the definition of inner and outer corners extends immediately using $D^{SH}_\mu$ and $D^{SH}_{\lambda/\mu}$.

For $T \in \ShSYT(\lambda/\mu)$ (not marked!) and $c_1 = (i_1,j_1)$ an inner corner of $\lambda/\mu$, we perform an \defn{inner jeu de taquin slide} using the following algorithm.

\smallskip

\nin
\qquad
{\sf While} \ts $c_k = (i_k,j_k)$ is not an outer corner:
\begin{itemize}
\item 	fill \ts $c_k$ \ts with \ts $\min T_{i+1\ j}, T_{i\ j+1}$,
\item   set \ts $c_{k+1} \gets (i+1,j)$ \ts if \ts $T_{i+1,j} \leq T_{i,j+1}$, and \ts $c_{k+1} \gets (i,j+1)$ otherwise.
\end{itemize}

\nin
\qquad
{\sf When} \ts $c_k$ is an outer corner, remove it from the resulting tableau, which we denote $J_{(i_1,j_1)}(T)$.

\smallskip

\nin
This definition applies equally well, mutatis mutandis, to shifted shapes.

\smallskip

For a tableau \ts $T \in \SSYT(\lambda/\mu)$,  choose \ts $S \in \SYT(\mu)$.  Let $c^k$
be the cell in \ts $D_\mu$, so that $S(c^k) = k$ for all \ts $1\le k \le |\mu|$.
The \defn{rectification} of $T$ is
\[
\rect(T) = J_{c^1} \circ \dots \circ J_{c^{|\mu|}}(T).
\]
For $T \in \ShSSYT(\lambda/\mu)$, define $\rect_{Sh}$ analogously.
It is known that $\rect(T)$ and $\rect_{Sh}(T)$ do not depend on the choice of~$S$, see e.g.~\cite{Hai2}.

\begin{example}{\rm
We compute a jeu de taquin slide with inner corner $c = (1,2)$:

\begin{align*}
T \ = \ \
	\raisebox{\height}{\begin{ytableau}
	\none & \none[\bullet] & 2 & 5 & 9\\
	\none & 2 & 4 & 7 & 10\\
	1 & 6 & 8
\end{ytableau}} \quad
\longrightarrow \ \
\raisebox{\height}{\begin{ytableau}
	\none & 2  & 2 & 5 & 9\\
	\none & \bullet & 4 & 7 & 10\\
	1 & 6 & 8
\end{ytableau}}
\quad
\longrightarrow \ \
\raisebox{\height}{\begin{ytableau}
	\none & 2  & 2 & 5 & 9\\
	\none & 4 & \bullet & 7 & 10\\
	1 & 6 & 8
\end{ytableau}}
\quad
\longrightarrow \ \
\raisebox{\height}{\begin{ytableau}
	\none & 2  & 2 & 5 & 9\\
	\none & 4 & 7 & \bullet & 10\\
	1 & 4 & 8
\end{ytableau}}
\quad
\longrightarrow \ \
\raisebox{\height}{\begin{ytableau}
	\none & 2  & 2 & 5 & 9\\
	\none & 4 & 7 & 10 & \none[\bullet]\\
	1 & 6 & 8
\end{ytableau}}
\quad = \ J_{(1,2)}(T)
\end{align*}

\smallskip

\nin
Here, the equalities ignore the sliding square which we denote by ``$\bullet$''.
In this case, the rectification of $T$ is given by
\[
\rect(T) \ = \ \
 \raisebox{\height}{\begin{ytableau}
	1 & 2  & 2 & 5 & 9\\
	4 & 7 & 10\\
	6 & 8
\end{ytableau}}
\]
}\end{example}

\smallskip

\begin{theorem}[{\cite[Thm~2.25]{Ser}}]
\label{t:rectification}
For every permutation $w$, we have \. $\QMS(w) = \rect_{SH}\bigl(Q(w)\bigr)$.
\end{theorem}

\medskip

\subsection{Reduced words for fully commutative permutations}
\label{ss:fc}
The symmetric group is generated by the simple transpositions $\{s_1,\dots, s_n\}$ with relations
\[
s_i s_j = s_j s_i \ \ \text{for} \ |i - j| > 1, \ \ \ s_i s_{i+1} s_i = s_{i+1} s_i s_{i+1}\ \ \ \mathrm{and} \ \ \ s_i^2 = 1.
\]
The first relation is called a \defn{commutation relation}, while the second is called a \defn{braid relation}.
For $w \in W$, we say that \ts $\bfa = (a_1,\dots,a_p)$ \ts is a \defn{reduced word} of $w$ if \ts
$w = s_{a_1} \dots s_{a_p}$ \ts and $p=\ell(w)$ is the number of inversions in~$w$.
Let $\mr(w)$ denote the set of reduced words of~$w$.
In this setting, the \emph{Matsumoto--Tits theorem} (see e.g.~\cite[Thm~25.2]{Bump}),
says that $\mr(w)$ is connected by commutation relations and braid relations.
A permutation $w$ is called \defn{fully commutative} if $\mr(w)$ is connected only by commutation relations.
We summarize a key result about fully commutative permutations from~\cite{BJS}.

\smallskip

\begin{theorem}[{\cite[$\S$2]{BJS}}]
\label{t:fc}
For each fully commutative permutation $w$, there is a skew shape \ts
$\sigma(w) = \lambda/\mu$ \ts and a bijection \ts $\FC:\SYT(\sigma(w)) \to \mr(w)$, so that
\[
\Des\bigl(\FC(T)\bigr) \. = \. \Des(T)^r.
\]
\end{theorem}

\smallskip

Our description of $\FC$ follows~\cite{Ste2}.
Put a partial order \ts $(I(w),<)$ \ts on the  \defn{inversion set}
\[
I(w) \. := \. \bigl\{(i,j)~:~i < j\,, \ w^{-1}(i) > w^{-1}(j)\bigr\}.
\]
We say that \ts $(i,j) \gtrdot (k,l) \in I(w)$ \ts if \ts $k = i$ \ts and \ts
$l = \min\{p > j: (j,p) \not\in I(w)\}$, or \ts $j = l$ \ts and \ts $i = \max\{p < k: (p,k) \not \in I(w)\}$.
The skew shape $\sigma(w)$ is now defined to be the poset $(I(w),<)$.
Each cell of $\sigma(w)$ corresponds to an inversion of~$w$, while
each \ts $T \in \SYT(\sigma(w))$ \ts provides an ordering of inversions for~$w$,
which corresponds to a reduced word of~$w$ as follows.
For $L$ a linear extension of \ts $(I(w),<)$, we define a word \ts
$\FC(L) = (a_1,\dots,a_p)$, where $a_i$ is the index so that
$$\prod^{i-1}_{j=1} \. L(j) \. s_{a_i} \, = \, \prod^{i}_{j=1} \. L(j)\..$$

\smallskip

\begin{example}{\rm
The permutations (given in one-line notation) $v = 241635$ and $w = 246135$ are fully commutative.
We can visualize their inversion sets \. $I(v) = \bigl\{(1,2), \. (1,4), \. (3,4), \. (3,6), \. (5,6)\bigr\}$
\. and \. $I(w) = I(v) \cup \bigl\{(1,6)\bigr\}$ \. using the diagrams
\begin{center}
\pic[scale = .4,yscale=-1]
\node at (0,1) (0/0) {1};
\node at (1,1) (1/0) {2};
\node at (2,1) (2/0) {3};
\node at (3,1) (3/0) {4};
\node at (4,1) (4/0) {5};
\node at (5,1) (5/0) {6};

\node at (0,-4) (0/-3) {2};
\node at (1,-4) (1/-3) {4};
\node at (2,-4) (2/-3) {1};
\node at (3,-4) (3/-3) {6};
\node at (4,-4) (4/-3) {3};
\node at (5,-4) (5/-3) {5};

\draw (0/0) to (0,0) to (2,-2) to (2/-3);
\draw (1/0) to (1,0) to (0,-1) to (0/-3);
\draw (2/0) to (2,0) to (4,-2) to (4/-3);
\draw (3/0) to (3,0) to (1,-2) to (1/-3);
\draw (4/0) to (4,0) to (5,-1) to (5/-3);
\draw (5/0) to (5,0) to (3,-2) to (3/-3);

\draw[red] (1.5,-2.5) to (-.5,-.5) to (.5,.5) to (1.5,-.5) to (2.5,.5) to (3.5,-.5) to (4.5,.5) to (5.5,-.5) to (3.5, -2.5);
\draw[red] (.5,-1.5) to (1.5,-.5) to (3.5,-2.5);
\draw[red] (1.5,-2.5) to (3.5,-.5) to (4.5,-1.5);
\epic
\hspace{20pt}
\pic[scale = .4,yscale = -1]
\node at (0,1) (0/0) {1};
\node at (1,1) (1/0) {2};
\node at (2,1) (2/0) {3};
\node at (3,1) (3/0) {4};
\node at (4,1) (4/0) {5};
\node at (5,1) (5/0) {6};

\node at (0,-4) (0/-3) {2};
\node at (1,-4) (1/-3) {4};
\node at (2,-4) (2/-3) {6};
\node at (3,-4) (3/-3) {1};
\node at (4,-4) (4/-3) {3};
\node at (5,-4) (5/-3) {5};

\draw (0/0) to (0,0) to (3,-3) to (3/-3);
\draw (1/0) to (1,0) to (0,-1) to (0/-3);
\draw (2/0) to (2,0) to (4,-2) to (4/-3);
\draw (3/0) to (3,0) to (1,-2) to (1/-3);
\draw (4/0) to (4,0) to (5,-1) to (5/-3);
\draw (5/0) to (5,0) to (2,-3) to (2/-3);

\draw[red] (-.5,-.5) to (.5,.5) to (1.5,-.5) to (2.5,.5) to (3.5,-.5) to (4.5,.5) to (5.5,-.5) to (2.5, -3.5) to (-.5,-.5);
\draw[red] (.5,-1.5) to (1.5,-.5) to (3.5,-2.5);
\draw[red] (1.5,-2.5) to (3.5,-.5) to (4.5,-1.5);
\epic
\end{center}
The corresponding shapes are \ts $\sigma(241635) = (3,2,1)/(1)$ \ts while \ts $\sigma(246135) = (3,2,1)$.

The reduced word $(1,3,2,5,4,3) \in \mr(w)$ \ts with inversion sequence \ts
$(1,2),(3,4),(1,4),(5,6), (3,6),(1,6)$ \ts corresponds to the linear order shown below.
Complementing the values, rotating $45^\circ$ counter-clockwise and transposing, we obtain
the desired $T\in \SYT(\si(w))$.
%
\[		\raisebox{-.5\height}{\includegraphics[height=1.5cm]{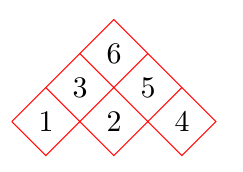}}\qquad \mapsto \qquad
       \raisebox{-.5\height}{ \includegraphics[height=1.5cm]{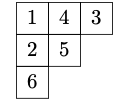}} = T. \]
One can check $\Des((1,3,2,5,4,3)) = \{2,4,5\}$ and $\Des(T) = \{1,2,4\} = \{2,4,5\}^r$.}
\end{example}

\smallskip

\begin{remark}
	{\rm The reason $\Phi$ reverses descent sets is that our conventions identify the skew shape $\lambda/\mu$ with the permutation $w$ such that $I(w)$ is the reverse of $D_{\lambda/\mu}$ as a poset.
	This choice is necessary for our proof outline to work.
	}
\end{remark}

\smallskip

A permutation with exactly one descent is called \defn{Grassmannian}.

\smallskip

\begin{proposition}
\label{p:grass}
	For $w$ Grassmannian, $\Red(w)$ is a single Knuth class.
\end{proposition}

\smallskip

While this proposition is implicit in~\cite{EG}, it can also be derived from $\FC$ by identifying the action of a Knuth move on $\bfa \in \Red(w)$ with the corresponding action on $\FC^{-1}(\bfa)$, which is a dual equivalence move as introduced in~\cite{Hai2}.
In either case, the proof is bijective.

\smallskip

Fix a Grassmannian permutation \ts $w^{\staircase_n} := 24\dots (2n) \. 13\dots (2n{-}1)$.
By identifying its inversions, it is clear that \ts $\sigma(w^{\staircase_n}) = \staircase_n$.

\smallskip

\begin{corollary}
\label{c:FC}
	The map~$\FC$ is a bijection from \ts $\SYT(\staircase_n)$ \ts to \ts
$\Red(w^{\staircase_n})$, and \ts $\Red(w^{\staircase_n})$ \ts  is a single Knuth class.
\end{corollary}

\smallskip

When $w$ is Grassmannian,~\Cref{p:grass} shows that \ts $\FC^{-1}$ is differs from
the RSK recording tableau by a simple transformation (see also~\cite{EG}).
One can extend this relationship to the skew setting.
However, our proofs in \Cref{s:K} require working with~$\FC$.
Additionally, note that~$\FC$ can be computed in $O(n\log n)$ operations when $|\lambda| = n$ , while computing
an insertion algorithm or its inverse requires \ts $O\bigl(\max\{\lambda_1,\lambda'_1\} \cdot n \log n\bigr)$
\ts operations, cf.~\cite{PV2}.


\bigskip

\section{Main results}
\label{s:main}

Let $\mu \subset \delta_n$ and $T \in \SYT(\delta_n/\mu)$.
The proofs in this section rely on the map $\vp = \QSW \circ \Phi$:
\begin{equation}
\label{eq:bij}
T \xrightarrow{\FC} \bfa \xrightarrow{\QSW} \QSW(\bfa).
\end{equation}
By \Cref{t:sw}~(3) and \Cref{t:fc},  $\vp$ reverses descent sets.
When $\mu = \rho_{\ell,m}$, one checks for an appropriate choice of $T$ that $\vp(T) \in \ShSYT'(\vr_n/\tau_{\ell,m})$.
By applying properties of Knuth classes (see Proposition~\ref{p:shK2}) and assuming Theorem~\ref{t:DeWitt-SYT}, this implies $\vp$ is the desired bijection.
In this section, we give an independent combinatorial proof that $\vp$ is a bijection.

\subsection{The bijection for staircases}

As a warmup, we show $\vp$ is a bijection between $\SYT(\delta_n)$ and $\ShSYT'(\vr_n)$.
Our proof relies on a simple observation that is used implicitly in~\cite[$\S$10]{HPPW}.

\smallskip

\begin{proposition}
	\label{p:shK}
	For any $n$, $\Red(w^{\delta_n})$ is a single shifted Knuth class.
\end{proposition}

\begin{proof}
By~\Cref{c:FC}, the result will follow by confirming that exchanging the first two entries of $\bfa \in \Red(w^{\delta_n})$ results in another element of $\Red(w^{\delta_n})$.
This follows from the definition of $\FC$, which implies in this case that the first two entries have the same parity so the exchange is a commutation relation.
\end{proof}

For a shifted shape $\lambda$, recall $M^\lambda$ is the minimal increasing tableau of shape $\lambda$ in the alphabet \ts $\{1,2,\dots\}$.
For example, 
\[
 \raisebox{\height}{\begin{ytableau}
 	1 & 2 & 3& 4\\
 	\none & 3 & 4 & 5\\
 	\none & \none & 5 & 6\\
 	\none & \none & \none & 7
 \end{ytableau}} \quad = \ M^{\vr_5}.
\]

\smallskip

\begin{lemma}\label{l:PSW}
	Let \ts $\bfa \in \Red(w^{\delta_n})$.
	Then \ts $\PSW(\bfa) \ts = \ts M^{\vr_n}$.
\end{lemma}

\begin{proof}
	By~\Cref{p:shK}, the result follows if it holds for some $\bfa \in \Red(w^{\delta_n})$.
	Verifying this fact is easy for the reduced word $(1,3,5,\dots, 2n{-}1,2,4\dots, 2n{-}2, 3,5, \dots, 2n{-}3,\dots) \in \Red(w^{\delta_n})$.
\end{proof}

\smallskip

\begin{proposition}
	\label{p:staircase}
	The map~$\vp$ is a bijection from \ts $\SYT(\delta_n)$ \ts to \ts $\ShSYT'(\vr_n)$.
Hence, \ts $s_{\delta_n} = P_{\vr_n}$.
\end{proposition}

\begin{proof}
By \Cref{c:fundamental} and \Cref{l:rev-comp}, the result will follow from a  bijection \ts $\SYT(\delta_n) \to \ShSYT'(\vr_n)$ that reverses descent sets.
By \Cref{p:staircase} and \Cref{l:PSW}, the map $\vp$ is a bijection.
\end{proof}

\smallskip

\medskip

\subsection{Proof of Theorem~\ref{t:skew}}
To prove~\Cref{t:skew}, we will need to describe the fully commutative permutation corresponding to $\delta_n/\mu$ and a small extension of~\Cref{p:shK}.

Given $\mu = (\mu_1,\dots,\mu_\ell) \subseteq \staircase_n$, let $w^{\staircase_n/\mu}$ be the permutation obtained from $w^{\staircase_n}$ by sequentially moving $(2i-1)$ to the left $\mu_i$ positions for $i \in [k]$.
For example, for $n=6$, we have \. $w^{\delta_6} \ts = \ts 2 \ts 4 \ts 6 \ts 8  \,\ts 10 \,\ts  1 \ts 3 \ts 5 \ts 7 \ts 9$ \. and \.
$w^{\delta_6/(3, 1)} \ts = \ts 2 \ts 4 \ts 1 \ts 6 \ts 8 \ts 3 \,\ts 10 \,\ts  5 \ts 7 \ts 9$.

\smallskip

\begin{lemma}
\label{l:skewFC}
	For all \ts $\mu \subseteq \staircase_n$, we have \ts $\sigma(w^{\staircase_n/\mu}) = \staircase_n/\mu$, \ts so \.  $\FC:\SYT(\staircase_n/\mu) \to \Red(w^{\staircase_n/\mu})$.
\end{lemma}

\begin{proof}
	The result follows from direct inspection.
\end{proof}

\smallskip

\begin{proposition} \label{p:shK2}
	Let \ts $\mu \subseteq \delta_n$.
	Then \ts $\Red\bigl(w^{\delta_n/\mu}\bigr)$ \ts is a union of shifted Knuth classes.
\end{proposition}

\begin{proof}
The proof is essentially identical to that of~\Cref{p:shK}, except that \ts $\Red\bigl(w^{\delta_n/\mu}\bigr)$ \ts is a union of Knuth classes, hence a union of (possibly fewer) shifted Knuth classes.	
\end{proof}

\begin{proof}[Proof of~\Cref{t:skew}]
Let $\mu \subseteq \delta_n$.
By~\Cref{p:shK2}, $\Red(w^{\delta_n/\mu})$ is a union of shifted Knuth classes.
Hence, by~\Cref{t:sw}~(2), the map~$\vp$ from~\eqref{eq:bij} is a bijection
\[
\SYT(\delta_n/\mu) \, \to \, \bigsqcup_{\lambda \in M_{\delta_n/\mu}} \ShSYT'(\lambda),
\]
for some multiset of partitions determined by $\delta_n/\mu$.
This gives a bijective proof that \ts $s_{\delta_n/\mu} \in \wt{\Lambda}$.

To prove \ts $s_{\delta_n/\mu} = s_{\delta_n/\mu'}$, we show that $\vp$,
as applied to both $\SYT(\delta_n/\mu)$ and $\SYT(\delta_n/\mu')$,
results in the same multisets of shifted standard tableaux with marked entries.
Indeed, let $\vp$ be the map for \ts $\SYT(\delta_n/\mu)$,
and let $\vp'$ be the map for \ts $\SYT(\delta_n/\mu')$.  Note that for \ts
$T \in \SYT(\delta_n/\mu)$ \ts with \ts $\FC(T) = (a_1,\dots,a_p)$, we have
\ts $\FC(T') = (n-a_1,\dots,n-a_p)$.
Therefore, $\Des(\FC(T')) = \Des(\FC(T))^c$.

Since $\QSW$ is invariant under shifting
each entry of a word, by~\Cref{t:neg} we have
\begin{equation}\label{e:stem}
T \ \xrightarrow{\vp} \ \vp(T) \ \xrightarrow{\circ} \ \vp(T)^\circ  \ \xrightarrow{(\vp')^{-1}} \ (\vp')^{-1}(\vp(T)^\circ)
\end{equation}
is a descent set preserving bijection from $\SYT(\delta_n/\mu)$ to $\SYT(\delta_n/\mu')$, so the result follows by~\Cref{c:fundamental}.
\end{proof}

\begin{proof}[Proof of~\Cref{c:skew-LR}]
Recall the RSK interpretation of LR--coefficients, see e.g.~\cite[Thm~A1.3.1]{EC} and~\cite{Zel}.
The desired bijection now follows combining our bijection in~\eqref{e:stem} with RSK and $\RSK^{-1}$.
The details are straightforward.
\end{proof}

\medskip

\subsection{Proof of Theorem~\ref{t:dewitt}}
\label{ss:bijection}
We outline the strategy with an example.
\begin{example}\label{ex:inverse}{\rm 	
Consider the shifted tableaux $T \in \ShSYT'(\vr_6 - \tau^{2,2}), \wt{T} \in \ShSYT'(\vr_6)$ such that
{\large
\[
T  \ = \ \  \begin{ytableau}
 1 & 2 & 4 & 6' & 9'\\
 \none & 3 & 5 & 8 & 11'\\
 \none & \none & 7 & 10'
 \end{ytableau}
 \qquad \mathrm{and} \qquad
\wt{T}  \ = \ \  \begin{ytableau}
 1 & 2 & 4 & 6' & 9'\\
 \none & 3 & 5 & 8 & {\ns 11'} \\
 \none & \none & 7 & {\ns 10'} & {\ns 13'}\\
 \none & \none & \none & {\ns 12} & {\ns 14} \\
 \none & \none & \none & \none & {\ns 15} 	
 \end{ytableau}\ .
\]
}
The Worley--Sagan inverse of $(M^{\vr_6},\wt{T})$ is $\wt{\bfa} = (1,7,5,9,8,3,6,7,2,4,3,5,4,6,5)$.
Removing the last four entries of $\wt{\bfa}$, we have $\bfa = (1,7,5,9,8,3,6,7,2,4,3)$ satisfying $\QSW(\bfa) = T$.
Viewing \ts $\wt{\bfa}$ \ts as a linear extension of its heap, we obtain:
\begin{figure}[hbt]
		\vspace{-.5cm}
\[
\pic[scale=.3]
\draw[red] (0,0) -- (5,5);
\draw[red] (0,0) -- (1,-1) -- (6,4);
\draw[red] (1,1) -- (3,-1) -- (7,3);
\draw[red] (2,2) -- (5,-1) -- (8,2);
\draw[red] (3,3) -- (7,-1) -- (9,1);
\draw[red] (4,4) -- (9,-1) -- (10,0);
\draw[red] (5,5) -- (10,0);
\node at (1,0) {1};
\node at (3,0) {6};
\node at (5,0) {3};
\node at (7,0) {2};
\node at (9,0) {4};

\node at (2,1) {9};
\node at (4,1) {10};
\node at (6,1) {7};
\node at (8,1) {5};

\node at (3,2) {11};
\node at (5,2) {12};
\node at (7,2) {8};

\node at (4,3) {13};
\node at (6,3) {14};

\node at (5,4) {15};

\epic
\]

\end{figure}

\vskip-.3cm

\nin
Now we can compute:

\[
\FC^{-1}(\wt{\bfa})  \ = \ \
\begin{ytableau}
	1 & 3 & 5 & 7 & 15\\
	2 & 4 & 6 & 10\\
	8 & 9 & 13\\
	11 & 14\\
	12
\end{ytableau}  \ \ \in \  \SYT(\delta_6/\rho^{2,2})
 \quad \mathrm{and} \quad
\FC^{-1}(\bfa)  \ = \ \
\begin{ytableau}
	\none & \none & 1 & 3 & 1\\
	\none & \none & 2 & 6\\
	4 & 5 & 9\\
	7 & 10\\
	8
\end{ytableau}  \ \ \in \  \SYT(\delta_6/\rho^{2,2})\..
\]

\smallskip

\nin
Therefore, \ts $\wt{T} = \vp\bigl(\FC^{-1}(\wt{\bfa})\bigr)$, and
our bijection maps \ts $\FC^{-1}(\bfa)$ to $T$.
}\end{example}

\smallskip

We generalize~\Cref{ex:inverse} as follows:

\begin{proposition}
\label{p:inverse}
For positive integers $\ell \leq m < n$ with $\ell + m < n$, there is an injection
\[
\psi \, : \ \ShSYT'\bigl(\vr_n - \tau^{\ell,m}\bigr) \ \longrightarrow \ \SYT\bigl(\delta_n/\rho^{\ell,m}\bigr)
\]
that reverses descent sets.
\end{proposition}

\begin{proof}
	Let \ts $T \in \ShSYT'(\vr_n - \tau^{\ell,m})$.
	Following~\Cref{ex:inverse}, construct $\wt{T}$ by adding to $T$ as follows:

\begin{equation}
	\label{eq:prime-tab}
{\footnotesize
%
\ytableausetup{nosmalltableaux}
\begin{ytableau}
	\none & \none & \none & b'\\
	\none & \none & \iddots & \vdots\\
	\none & a' & \dots & c'\\
	& & & \\
	\vdots & \vdots & \vdots & \vdots\\
	& & & \\
	\none[a] & & & \\
	\none[\vdots] & \none[\ddots] &\ddots  & \vdots \\
	\none[b] & \none[\dots] & \none[c] &
\end{ytableau}
}\end{equation}

\smallskip

Fill the entries of $\rho^{m,\ell}$ column by column from left to right, top to bottom, with values from \ts
$\bigl\{\binom{n}{2} - \ell \cdot m+1,\ldots,\binom{n}{2}\bigr\}$, in increasing order.
Reflect the cutout staircase $\rho^{\ell,m} \setminus \tau^{\ell,m}$, marking each entry  and place the resulting tableau on top of $\rho^{\ell,m} \setminus \tau^{\ell,m}$.
The values $a,b,c$ in~\Cref{eq:prime-tab} show where corresponding entries are mapped by this operation.
Computing the Worley--Sagan inverse of $(M^{\vr_n},\wt{T}^\circ)$, we obtain some $\wt{\bfa} \in \Red(w^{\delta_n})$.
Removing the last $\ell\cdot m$ entries of $\wt{\bfa}$, we obtain~$\bfa$.
Now define \ts $\psi(T) := \bfa$.
Since map~$\psi$ is a restriction of $\vp^{-1}$, we see it reverses descent sets.
What remains is to show the image of $\psi$ lies in \ts $\SYT(\delta_n/\rho^{\ell,m})$.

Let \ts $a := m - \ell$.
We claim the last $\ell \cdot m$ entries $\bfa$ are necessarily of the form
\begin{equation}\label{eq:word}
(n + \ell - 1, n + \ell - 2, \dots, n - a - 1, n + \ell- 2, n +\ell - 3, \dots, n-a-2, \dots , n + m - 1, \dots, n - 1).	
\end{equation}
The proof is a straightforward induction using inverse Worley--Sagan insertion.
Rather than give complete details, we demonstrate the result with an example with
\ts $\ell = 2$, $m = 3$, and \ts $n = 6$.
Here we have:
\[\ytableausetup{smalltableaux}
M^{\vr_n} \ = \ \ \raisebox{\height}{\begin{ytableau}
 	1 & 2 & 3& 4 & 5\\
 	\none & 3 & 4 & 5 & 6\\
 	\none & \none & 5 & 6 & 7\\
 	\none & \none & \none & 7 & 8\\
 	\none & \none & \none & \none & 9
 \end{ytableau}}
 \qquad \mathrm{and} \qquad
 \wt{T}^\circ - T^\circ  \ = \ \
 \raisebox{\height}{\begin{ytableau}
 	 \cdot & \cdot  & \cdot & \cdot & \cdot \\
	\none & \cdot & \cdot & \cdot & 12'\\
 	\none & \none & \cdot & 10 & 13\\
 	\none & \none & \none & 11 & 14\\
 	\none & \none & \none & \none & 15
 \end{ytableau}}
\]
When inverting the insertions that add the $15, 14$ and $13$, we see they must come from row insertions with all bumps in the last column.
After this, we obtain
\[
M^{\vr_n} \ = \ \ \raisebox{\height}{\begin{ytableau}
 	1 & 2 & 3& 4 & 8\\
 	\none & 3 & 4 & 5 & 9\\
 	\none & \none & 5 & 6\\
 	\none & \none & \none & 7
 \end{ytableau}}
 \qquad \mathrm{and} \qquad
 \wt{T}^\circ - T^\circ \ = \ \
 \raisebox{\height}{\begin{ytableau}
 	 \cdot & \cdot  & \cdot & \cdot & \cdot \\
	\none & \cdot & \cdot & \cdot & 12'\\
 	\none & \none & \cdot & 10 \\
 	\none & \none & \none & 11 \\
 \end{ytableau}}
\]
Inverting the insertion that added $12'$, we must use column insertion.
Necessarily, the $9$ was bumped by the $7$, which was row bumped by the entries above it in the fourth column.
This results in:
\[
M^{\vr_n}  \ = \ \  \raisebox{\height}{\begin{ytableau}
 	1 & 2 & 3& 5 & 8\\
 	\none & 3 & 4 & 6\\
 	\none & \none & 5 & 7\\
 	\none & \none & \none & 9
 \end{ytableau}}
 \qquad \mathrm{and} \qquad
 \wt{T}^\circ - T^\circ  \ = \ \
\raisebox{\height}{ \begin{ytableau}
 	 \cdot & \cdot  & \cdot & \cdot & \cdot \\
	\none & \cdot & \cdot & \cdot \\
 	\none & \none & \cdot & 10 \\
 	\none & \none & \none & 11 \\
 \end{ytableau}}
\]
The remaining entries necessarily came from insertions occurring in the fourth column as well, resulting in
\[
M^{\vr_n}  \ = \ \  \raisebox{\height}{\begin{ytableau}
 	1 & 2 & 3& 7 & 8\\
 	\none & 3 & 4 & 9\\
 	\none & \none & 5
 \end{ytableau}}
 \qquad \mathrm{and} \qquad
 \wt{T}^\circ - T^\circ  \ = \ \
 \raisebox{\height}{\begin{ytableau}
 	 \cdot & \cdot  & \cdot & \cdot & \cdot \\
	\none & \cdot & \cdot & \cdot \\
 	\none & \none & \cdot
 \end{ytableau}}
\]
The last values of $\bfa$ in this example are $(6,5,4,7,6,5)$, as claimed in~\eqref{eq:word}.
The reader should be able to extend this example to an inductive proof of our claim with little difficulty but some tedium.
\end{proof}

Assuming~\Cref{t:dewitt}, we see~\Cref{p:inverse} gives a bijection from \ts $\ShSYT'(\vr_n - \tau^{\ell,m})$
\ts to \ts $\SYT(\delta_n/\rho^{\ell,m})$.
Since $\psi = \vp^{-1}$ when $\vp(T) \in \ShSYT'(\vr_n - \tau^{\ell,m})$, we can conclude $\vp$ is the desired bijection.
However, it remains to prove this fact directly without resorting to~\Cref{t:dewitt}.
In principle, one could have some $\bfa \in \Red(w^{\delta_n})$ that ends with the values in~\eqref{eq:word} such that $\QSW(\bfa)$ does not contain $\wt{T} - T$ as a subtableau.

Rather than prove directly that $\vp$ has the desired image, we outline a map essentially equivalent to $\vp$ using RSK, mixed shifted insertion and  jeu de taquin (this map could also be used for sampling).
For $\lambda$ a partition, let $S^\lambda \in \SYT(\lambda)$ be the \defn{superstandard tableau}, which is the unique standard tableau whose rows are consecutive.
Computing slides row-by-row, we observe:
\begin{equation}
	\label{eq:rect}
	S^{\vr_n} \, = \, \rect\bigl(S^{\delta_n}\bigr)\..
\end{equation}

\smallskip

The following lemma is implied by the proof of~\Cref{p:inverse} and~\Cref{t:mixed}~(2).

\smallskip

\begin{lemma}
\label{l:first-part}
	Let $T \in \SYT'(\tau_{r,\ell})$ be obtained from the tableau in~\eqref{eq:prime-tab} by reflecting across the line $y=x$ and complementing values.
	Then $T = \QSW\bigl(r(S^{\rho_{r,\ell}})\bigr)$.
\end{lemma}

We now outline the alternate bijection to $\vp$:\footnote{Strictly speaking, it follows from~\Cref{t:dewitt} that our map~$\psi$ in~\Cref{p:inverse} is bijective.
The argument here is employed to obtain a self-contained proof.}

\smallskip

\begin{enumerate}
	\item For $P_0 \in \SYT(\delta_n/\rho_{r,\ell})$, complete it to $\wt{P_0} \in \SYT(\delta_n)$ with the same relative order \\
in \ts $D_{\delta_n/\rho_{r,\ell}}$ \ts so that \ts $\wt{T}\mid\rho_{r,\ell} = S^{\rho_{r,\ell}}$
	\item Let \ts $w := \RSK^{-1}\bigl(\wt{T},S^{\delta_n}\bigr)$
	\item By~\Cref{t:rectification} and Equation~\eqref{eq:rect}, the mixed shifted insertion maps \ts $w$ \ts
to \ts $\bigl(\PMS(w),S^{\vr_n}\bigr)$
	\item Flip \ts $\PMS(w)$ \ts across the \ts $y=x$ \ts line and complement its values to obtain \ts $\wt{P_1}$
	\item Define \ts $P_1 := \wt{P}_1 \mid D_{\vr_n - \tau_{r,\ell}}$
\end{enumerate}

\smallskip

\nin
By~\Cref{l:first-part}, we see the entries in $\wt{P}_1 - P_1$ are the same as those
in the tableau constructed in~\eqref{eq:prime-tab}.  Since each step outlined above is injective,
we obtain an injection from $\SYT(\delta_n/\rho_{r,\ell})$ to $\ShSYT'(\vr_n/-\tau_{r,\ell})$.

Note that Step~(3) is not obviously bijective, since~\Cref{t:rectification} only implies the forwards direction.
This completes our bijective proof of~\Cref{t:dewitt}, hence also of~\Cref{t:DeWitt-SYT}. \ \hfill $\sq$

\medskip

\subsection{Proof of Theorem~\ref{t:DeWitt-random}}
Start by generating a random \ts $T\in \ShSYT'\bigl(\vr_k - \tau^{\ell,m}\bigr)$ \ts
using either~\cite{Fis} or~\cite{Sag-GNW}.  Each algorithm involves $O(k^2)$ iterations
of $O(k)$ steps, each step involving moving a single square, giving the total $O(k^3)$ cost.
Use the algorithm above to compute \ts $\vp(T) \in \SYT\bigl(\delta_k/\rho^{\ell,m}\bigl)$.
The cost of the Worley--Sagan insertion is equal to that of RSK, thus bounded by
\ts $O(k^3\log k)$ \ts in this case~\cite{PV2}.  This implies the total bound as
in the theorem. \ \hfill $\sq$


\bigskip

\section{\texorpdfstring{$K$}{K}-theoretic extensions}\label{s:K}

\subsection{\texorpdfstring{$K$}{K}-theoretic objects}
The objects and maps introduced in \Cref{s:defns} have $K$-theoretic analogues.
We give compact descriptions of these objects, referring the reader to our references for concrete examples.

For $\lambda$ a partition, a \defn{set-valued standard Young tableau} of shape $\lambda$ and \defn{size} $n$ is a function $T:D_\lambda \to 2^{[n]}$, so that \ts $T(i,j)\ne\varnothing$,
\ts $\max T(i,j) < \min T(i{+}1,j), \min T(i,j{+}1)$ \ts and \ts $\sqcup_{(i,j) \in D_\lambda} T(i,j) = [n]$.
Note the latter condition requires the entries of $T$ to be disjoint.
Let $\Set(\lambda)$ be the set-valued tableaux of shape $\lambda$ and \ts $\Set_n(\lambda)$ be the subset of tableaux whose size is $n$.

Similarly, for $\mu$ a shifted shape, one defines \defn{shifted set valued standard tableaux}, denoted by $\ShSet(\mu)$ and $\ShSet_n(\mu)$, as well as \defn{marked shifted set valued standard tableaux}, denoted by \ts $\ShSet'(\mu)$ \ts and \ts $\ShSet'_n(\mu)$.  In the latter case, when $i > j$ each value in $T(i,j)$ is marked or unmarked individually.
For $|\lambda| = n$, note that \ts $\Set(\lambda) = \SYT(\lambda)$.
Likewise, if $|\mu| = m$ then $\ShSet_m(\mu) = \ShSYT(\mu)$ and $\ShSet'(\mu) = \ShSYT^2(\mu)$.

For $T \in \Set_n(\lambda)$, the \emph{descent set} is $\Des(T) := \{i \in [n-1]: i\ \mbox{is strictly above}\ i+1\} $
as before.
Similarly, for $U \in \ShSet'_n(\mu)$, the \emph{descent set} of $U$ is defined the same as for marked shifted standard tableaux.
Let $U^\circ$ be the tableau obtained from $U$ by marking every unmarked value and unmarking every marked value in off-diagonal entries of $U$.
It is easy to see that~\Cref{l:complement} extends to the set-valued setting.

\begin{lemma}
\label{l:Kdes}
	For $\lambda$ a shifted shape and $T \in \ShSet'_n(\lambda)$,  we have:  \ts $\Des(T^\circ) = \Des(T)^c$.
\end{lemma}

For $S \subset [n-1]$, recall the definition of $I_S$ from~$\S$\ref{ss:des}.
Define
\begin{equation}
	\label{eq:grothendiecks}
G_\lambda \, := \, \sum_{T \in \Set(\lambda)} \, \sum_{\bfi \in I_{\Des(T)}} \. x^\bfi \qquad \mbox{and}
\qquad GP_\mu \, := \, \sum_{T \in \ShSet'(\mu)} \, \sum_{\bfi \in \Des(T)} \. x^\bfi\..
\end{equation}
These definitions are non-standard, differing from the standard definitions by the invertible substitution of variables \ts
$x_i \mapsto \ts \frac{-x_i}{1-x_i}$ \ts and a factor of \ts $(-1)^{|\lambda|}$, which is more or less an application of the $\omega$ involution for symmetric functions; see~\cite[Thm~6.11]{PP3},~\cite[Thm~1.3]{HKPWZZ} and~\cite[Cor.~6.6]{LM}.
The tableaux \ts $T$ \ts arising from the summations of~\eqref{eq:grothendiecks} are multiset-valued semistandard tableaux,
while the standard definition is in terms of set-valued semistandard tableaux.

The $K$-theoretic analogue of Worley--Sagan insertion is called \defn{shifted Hecke insertion}, introduced in~\cite{PP2}.
Rather than define this map, we refer the reader to~\cite[$\S$2.2]{HKPWZZ} for a verbose definition,
or~\cite[$\S$5.2]{HMP3} for a pseudocode description.
An implementation of shifted Hecke insertion is available at~\cite{code}.

The \defn{$K$--Knuth relations} are the transformations:
\[
acb \leftrightarrow cab,\quad bac \leftrightarrow bca,\quad aba \leftrightarrow bab, \quad aa \leftrightarrow a \quad \mbox{with} \quad a < b < c.
\]
For a word \ts $\bfa = (a_1,a_2,\dots,a_p)$, a \defn{$K$--Knuth move} is an application of a $K$--Knuth relation, while a \defn{weak $K$--Knuth move} is a $K$--Knuth move  or the exchange of the first two entries $a_1$ and $a_2$.
Two words $\bfa$ and $\bfb$ are \defn{$K$--Knuth equivalent} if they differ by a sequence of $K$--Knuth moves and \defn{weak $K$--Knuth equivalent} if they differ by a sequence of weak $K$--Knuth moves.
These notions of equivalence are introduced in~\cite{BS} as $K$-theoretic analogues of Knuth and shifted Knuth equivalence.

For $\mu$ a shifted shape, let $\INC(\mu) \subseteq \ShSSYT(\mu)$ be the subset of $T$ whose rows are strictly increasing.

\smallskip

\begin{theorem}\label{t:shifted-Hecke}
\
\begin{enumerate}
	\item  \cite[Thm~5.19]{PP2} \, Shifted Hecke insertion is a bijection from words of length $n$ to pairs $(\PSH,\QSH)$ of tableaux where $\PSH \in \INC(\mu)$ and $\QSH \in \ShSet'_n(\mu)$ for some shifted shape $\mu$.
	\item \cite[Cor.~2.18]{HKPWZZ} \text{\rm and} \cite[Cor.~7.2]{BS} \, If $\PSH(\bfa) = \PSH(\bfb)$, then $\bfa$ and $\bfb$ are shifted $K$--Knuth equivalent.
	The converse holds when $\PSH(\bfa)$ is a minimal increasing tableau.
	\item \cite[Prop.~2.24]{HKPWZZ} \, For every word $\bfa$, we have \ts $\Des(\bfa) = \Des(\QSH(\bfa))$.
\end{enumerate}

\end{theorem}

\smallskip

\begin{remark}{\rm
Note~\Cref{t:shifted-Hecke}~(2) is a weaker statement than~\Cref{t:sw}~(2).
An analogue of~\Cref{t:neg} should exist for shifted Hecke insertion, but does not appear in the literature.
There is a $K$-theoretic analogue of jeu de taquin~\cite{TY2,CTY} that we do not require.
At present, there is no $K$-theoretic analogue of mixed shifted insertion.
}\end{remark}

\smallskip

\nin
\textbf{Hecke expressions.}
The \defn{$0$-Hecke monoid} $(S_n,\circ)$ of the symmetric group replaces the relation $s_i^2 = 1$ with the relation $s_i \circ s_i = s_i$.
For $w$ a permutation, we say $\bfa = (a_1,\dots,a_p)$ is a \defn{$0$-Hecke expression} for $w$ if $w = s_{a_1} \circ \dots \circ s_{a_p}$.
Let $\mh(w)$ be the set of $0$-Hecke expressions for $w$ and $\mh_n(w)$ be the subset of expressions of length $n$.
For $p = \ell(w)$, note $\mh_p(w) = \Red(w)$.

\smallskip

\begin{theorem}[{\cite[Prop.~14]{MPPS}, see also~\cite[\S 3]{Yun}}]
\label{t:res}
	For each fully commutative permutation $w$, there is a bijection $\res:\Set_n(\sigma(w)) \to \mh_n(w)$ so that
	\[
	\Des(\res(T)) = \Des(T)^r.
	\]
\end{theorem}

\smallskip

We should mention that the descent set relationship is not stated explicitly in~\cite{MPPS},
but follows immediately from their proof.
The construction of $\res$ parallels that of~$\Phi$.
Let $w$ be fully commutative and let $\bfa = (a_1,\dots,a_p) \in \mh(w)$.  For each \ts $i \in [p]$,
s.t.
\[
s_{a_1} \circ \dots  \circ s_{a_{i-1}} \. \neq \. s_{a_1} \circ \dots \circ s_{a_i}\.,
\]
place $i$ in the same entry as would be done in~$\Phi$.
Similarly, for each \ts $i \in [p]$, s.t.
\[
s_{a_1} \circ \dots  \circ s_{a_{i-1}} \. = \. s_{a_1} \circ \dots \circ s_{a_i}\.,
\]
there exists maximal $h < i$ with $a_h = a_i$; place $i$ in the same cell as~$h$ in this case.

\smallskip

Finally, there are $K$-theoretic analogues of~\Cref{p:grass} and~\Cref{c:FC}.

\smallskip

\begin{proposition}
	\label{p:Kgrass}
	For $w$ Grassmannian, the set \ts $\mh(w)$ \ts is a single $K$--Knuth class.
	In particular, the set \ts $\mh(w^{\delta_n})$ \ts a single $K$--Knuth class.
\end{proposition}

\smallskip

A combinatorial proof of the proposition is implicit in~\cite{BS}.

\smallskip

\subsection{\texorpdfstring{$K$}{K}-theoretic results.}
In this section, we explain how to extend~\Cref{t:DeWitt-SYT} to set-valued tableaux.
Along the way, we prove $K$-theoretic extensions of~\Cref{t:skew,t:skew-AS,t:dewitt}.
Our construction is essentially the same as $\vp$ with the necessary $K$-theoretic substitutions:
\[
\overline{\vp} \ :=\quad
T \, \xrightarrow{{\res}} \, \bfa \, \xrightarrow{{\QSH}} \, \QSH(\bfa).
\]
The proofs from~\Cref{s:main} will extend almost verbatim.
By~\Cref{t:res}, and \Cref{t:shifted-Hecke}~(3), we see that
\ts $\overline{\vp}$ \ts reverses descent sets.

To begin, we extend \Cref{p:staircase} and \Cref{t:skew,t:skew-AS} to the $K$-theoretic setting.

\begin{proposition}
\label{p:Kstaircase}
The map \. $\overline{\vp}$ \. is bijection: \. $\Set_m(\delta_n) \ts \to \ts \ShSet_m'(\vr_n)$ that reverses descent sets.
\end{proposition}

\begin{proof}
	By~\Cref{p:Kgrass}, the set \ts $\mh(w^{\delta_n})$ is a single $K$--Knuth equivalence class.
	Moreover, the first two entries in each \ts $\bfa \in \mh(w^{\delta_n})$ \ts are odd, so it is a single weak $K$--Knuth equivalence class.
	To see that \ts $\overline{\vp}$ \ts is a bijection, apply~\Cref{t:shifted-Hecke}~(1) and~(2).
Finally, observe that \ts $\PSH(\bfa)$ \ts is a minimal increasing tableau for some $\bfa \in \mh(w^{\delta_n})$
(this was already shown in the proof of~\Cref{p:staircase}).
\end{proof}

\smallskip

Denote by \ts$\overline{\Lambda}$ \ts the ring generated by \ts $GP_\nu$'s, see~\cite{IN}.

\smallskip

\begin{theorem}[{\cite[Thm.~6.7]{LM}}]
\label{t:Groth-skew}
	Let $\mu \subseteq \delta_n$.
	Then \ts $G_{\delta_n/\mu} \in \overline{\Lambda}$.
	Moreover, we have \. $G_{\delta_n/\mu} \ts = \ts \sum_{\nu} \ts b_\nu \. GP_\nu$ \.
where each \ts $b_\nu$ \ts is a non-negative integer.
\end{theorem}

\begin{proof}
By~\Cref{t:res}, we have:
\[
G_{\delta_n/\mu} \, = \, \sum_{n \geq 1} \. \sum_{\bfa \in \mh_n(w^{\delta_n/\mu})}
\. \sum_{\bfi \in I_{(\Des(\bfa)^c)^r}} \. x^\bfi \..
\]
Since \ts $\mh\bigl(w^{\delta_n/\mu}\bigr)$ \ts is closed under $K$--Knuth moves, it is a finite union of $K$--Knuth classes.
Additionally, the first two entries of \ts $\bfa \in \mh\bigl(w^{\delta_n/\mu}\bigr)$ \ts must both be odd.
This implies \ts $\mh\bigl(w^{\delta_n/\mu}\bigr)$ \ts is a union of (possibly fewer) weak $K$--Knuth classes.
Therefore by~\Cref{t:shifted-Hecke}~(2), we see $G_{\delta_n/\mu}$ is a sum of finitely many
\ts $GP_\mu$'s, and the result follows.
\end{proof}

\smallskip

\begin{corollary}[{\cite[Thm~1.3]{AMAN}}]
\label{c:AMAN}
	For all \ts $\mu \subseteq \delta_n$\., we have \ts $G_{\delta_n/\mu} \ts = \ts G_{\delta_n/\mu'}$\..
\end{corollary}

\begin{proof}
	By~\Cref{t:Groth-skew}, both \ts $G_{\delta_n/\mu}$ \ts and \ts
$G_{\delta_n/\mu'}$ \ts can be expressed as generating functions over
marked shifted semi-standard set-valued tableaux, which have the descent-set complementing involution $\circ$.
	Therefore, the result follows by observing that \ts
$\bfa = (a_1,\dots,a_n) \ts \mapsto \ts (n-a_1,\dots,n-a_n)$ \ts
is descent-set complementing bijection \. $\mh\bigl(w^{\delta_n/\mu}\bigr) \to \mh\bigl(w^{\delta_n/\mu'}\bigr)$.
\end{proof}

Next, we give an algebraic proof for a $K$-theoretic analogue of \Cref{t:dewitt}, which we will use to prove the $K$-theoretic analogue of \Cref{t:DeWitt-SYT}.

\begin{theorem}
	\label{t:K-dewitt}
	For $\ell + m < n$, $G_{\delta_n/\rho_{\ell,m}} = GP_{\vr_n - \tau_{\ell,m}}$.
\end{theorem}
\begin{proof}
	Combining~\cite[Cor~6.22]{HMP2} and~\cite[Cor.~4.6]{MP1}, we see $G_{\delta_n/\mu}$ is a single sympletic stable Grothendieck polynomial for some fixed-point-free involution $y$.
	By Theorem~\ref{t:dewitt} we see $s_{\delta_n/\rho_{\ell,m}} = P_{\vr_n - \tau_{\ell,m}}$, so $y$ is FPF-vexillary in the sense of~\cite{HMP3}.
	The result now follows from~\cite[Cor.~3.11]{MP2} and~\cite[Cor.~5.9]{HMP4},
since with each application of Corollary~3.11 only one term occurs on the RHS after cancelling like terms.
\end{proof}

\begin{theorem}
	\label{t:set-dewitt}
	The map~$\.\overline{\vp}\ts:\. \Set_m(\delta_n/\rho_{\ell,m}) \. \longrightarrow \. \ShSSet'_m(\vr_n-\tau_{\ell,m})$ \. is a bijection.
\end{theorem}

\begin{proof}
	We construct the inverse map as before.
	Fix \ts $T \in \ShSet'_m(\vr_n - \tau_{\ell,m})$.  Construct \ts $\wt{T}$ \ts
by filling entries in \ts $\vr_n/(\vr_n - \tau_{\ell,m})$ \ts as described in~\eqref{eq:prime-tab},
and let $P_1$ be the minimal increasing tableau of shape~$\vr_n$.
	One can easily check the first $\ell \cdot m$ steps to invert shifted Hecke insertion will coincide with those used in inverting Worley-Sagan insertion.
	Similarly, applying $\res^{-1}$ will then result in a tableau $\wt{T}' \in \Set(\delta_n)$ that, when restricted to $D_{\rho_{\ell,m}}$, gives the super standard tableau of that shape.
	The forward direction now follows by \Cref{t:DeWitt-SYT,t:K-dewitt}.
	\end{proof}

For geometric reasons, it is easier to work with symplectic stable Grothendieck polynomials in our proof of \Cref{t:K-dewitt}.
This is the same identification used by Lewis and Marberg in their proof of~\Cref{t:Groth-skew}.
However, we define $\overline{\vp}$ using shifted Hecke insertion rather than its symplectic analogue~\cite{Mar},
so that it is manifestly a generalization of $\vp$.

\bigskip

\section{Final remarks}\label{s:finrem}

\subsection{}  \label{ss:finrem-marv}
In~\cite{Pur}, Purbhoo constructs a bijection $\SYT(\delta_n) \to \ShSYT'(\vr_n)$
via a jeu de taquin-like algorithm called \emph{conversion}.
By~\cite[Prop.~7.1]{Hai1}, his map is equivalent to the alternate bijection
we define after \Cref{l:first-part}, which has a conversion formulation
in full generality.

\subsection{}  \label{ss:finrem-hist}
John Stembridge proved~\Cref{t:skew} in June~2004, but the proof
was never published~\cite{Ste3}.  Two proofs of~\Cref{t:skew}
were given in~\cite[Cor.~7.32]{RSS}, where the authors also attributed
this result to Stembridge.   A different algebraic proof was
was given in~\cite[Solution to Exc.~2.9.25]{GR}.  Recently, a generalization
of the theorem to stable Grothendieck polynomials is given
in~\cite[Cor.~\ref{c:AMAN}]{AMAN}. A different generalization
to \emph{Macdonald's ninth variation Schur functions} was given
in \cite{FoKi}; the proof is based on the {\em Hamel--Goulden identities}.

\subsection{}  \label{ss:finrem-det}
The approach in~\cite{KS} to the proof of~\Cref{c:dewitt}
is based on explicit computation of determinants.  This type of argument
somewhat hides the role of the staircase shape which is crucial for the
proof.  In a forthcoming paper~\cite{LMP} the authors extend the determinant
approach from \ts $\de_k/(b^a)$ \ts to all \ts $\de_k/\mu$ \ts using certain
new determinantal and $q$-determinantal identities.

\subsection{} \label{ss:finrem-crystal}
Let \ts $\mathfrak{q}(n)$ \ts be the \emph{queer Lie superalgebra}.  There is a \ts
$\mathfrak{q}(n)$-crystal structure on the set of words \ts $[n]^m$ \ts of
length~$m$ on the alphabet \ts $[n]$.  The semistandard version of \Cref{t:mixed}~(1)~\cite[Def.~1.2]{Ser} gives
\begin{align*}
& \mathrm{MS}: [n]^m \,  \to \, \bigsqcup_{\lambda \ts \vdash \ts m} \. \ShSSYT'(\lambda) \times \ShSYT(\lambda)\.,
\\
& \qquad  \text{where} \qquad w \, \to \, \bigl(P_{\mathrm{MS}}(w),Q_{\mathrm{MS}}(w)\bigr).
\end{align*}
The connected component of the crystal in which a word \ts $w \in [n]^m$ \ts
is found is specified by \ts $Q_{\mathrm{MS}}(w)$, giving a \ts
$\mathfrak{q}(n)$-crystal structure on shifted marked semistandard tableaux~\cite[Thm~3.2]{Hiro} (see also~\cite{HPS}).
Using the obvious correspondence between a partition without a staircase and
the complementary shape in a larger square, by~\cite[Prop.~23]{CK}, a \ts
$\mathfrak{q}(n)$-crystal structure on the set of semistandard skew tableaux
with maximum entry $n$ of shape staircase minus rectangle is obtained by
sending a tableau $T$ to a word \ts $w_T \in [n]^m$ \ts by scanning
rows from top to bottom, with each row read from right to left and
where a box with the entry $i$ is recorded as \ts $n+1-i$.

By uniqueness of crystals, it follows that a bijection from
skew semistandard tableaux to shifted marked semistandard tableaux
(each with entries bounded by $n$) is given by \ts $P_{\mathrm{MS}}(w_T)$,
where $w_T$ is the reading word (as above) of a skew tableau~$T$.
The bijective correspondence of~\Cref{t:DeWitt-SYT} follows by
restricting to the zero-weight space when $n=m$
(containing tableaux with standard content).

\begin{example}{\rm
Continuing with~\Cref{ex:inverse}, for
\[T\ = \, \raisebox{0.5\height}{\begin{ytableau}
	\none & \none & 1 & 3 & 11\\
	\none & \none & 2 & 6\\
	4 & 5 & 9\\
	7 & 10\\
	8
\end{ytableau}} \]
we construct the reading word \ts $w_T =(1,9,11,6,10,3,7,8,2,5,4)$, which
mixed inserts to the pair
\[\bigl(P_{\mathrm{MS}}(w_T),Q_{\mathrm{MS}}(w_T)\bigr) \, = \, \left(\raisebox{0.5\height}{\begin{ytableau}
 1 & 2 & 4 & 6' & 9'\\
 \none & 3 & 5 & 8 & 11'\\
 \none & \none & 7 & 10'
 \end{ytableau}} \ \ , \ \raisebox{0.5\height}{\begin{ytableau}
 1 & 2 & 3 & 7 & 8\\
 \none & 4 & 5 & 9 & 10\\
 \none & \none & 6 & 11
 \end{ytableau}}\right).\]  Restricting to $P_{\mathrm{MS}}(w_T)$ gives the desired bijection; it is possible to explicitly characterize $Q_{\mathrm{MS}}(w_T)$.  Note that this bijection works equally well when starting with a skew semistandard tableau instead of a skew standard tableau.
}\end{example}

\subsection{}  \label{ss:finrem-LR}
For symmetries of LR--coefficients, see~\cite{BZ,HS}. See also a bijective
proof in~\cite{PV} relating
the highly symmetric \emph{{\rm BZ}--triangles} and the (usual) LR--tableaux.
In summary, all these hidden symmetries of LR--tableaux have now been
established via a chain of bijections.
See also an unusual construction in~\cite{TY1} which trades off
effectiveness of a combinatorial interpretation for greater symmetry.
Finally, we refer to~\cite{PP1} for a brief overview of further examples
of hidden symmetry.

\subsection{}\label{ss:finrem-random}
Random generation (sampling) of combinatorial objects from the (exactly) uniform
distribution is a classical problem in both Combinatorics, see e.g.~\cite{NW},
and Theoretical Computer Science, see e.g.~\cite{JVV}.  The approach of using
determinantal formulas for uniform random generation
of planar structures was introduced by Wilson~\cite{Wil}.

For Young tableaux of staircase minus rectangle shape, our approach is
also greatly superior to the MCMC approach for the \emph{nearly uniform}
generation of linear extensions of all posets.  Indeed, the best known
general bound is $O(n^3\log n)$ time for $n$-element posets,
due to Bubley and Dyer~\cite{BD}.  In our case, we have \ts
$n=\Theta(k^2)$, giving only a \ts $O(k^6\log k)$ \ts time, which is
much weaker compared with the \ts $O(k^3\log k)$ \ts time in~\Cref{t:DeWitt-random}.
It would be interesting to see if a nearly linear MCMC algorithm can be obtain
for nearly uniform sampling from \ts $\SYT(\la/\mu)$ \ts in our case, or
(even better) for general skew shapes.

Finally, we should mention a detailed complexity analysis of the
\emph{NPS algorithm} given in~\cite{SS}; similar results likely
hold for Fischer's algorithm~\cite{Fis}.  Note that a rough \ts $O(k^3)$ \ts
suffices for our purposes as the cost of the algorithm is (marginally)
dominated by the Worley--Sagan insertion.

\subsection{}\label{ss:finrem-K}~\Cref{t:Groth-skew} does not
have the same applications as~\Cref{t:DeWitt-random}, since there
is no known probabilistic algorithm to sample from \ts
$\ShSSet'(\vr_n-\tau_{\ell,m})$ \ts uniformly at random, or
even nearly-uniformly.  In fact, we are only aware of few
incremental results for the number of certain set-valued
tableaux in some nice special cases, see e.g.~\cite{Dru,RTY}.
In particular, no determinantal formula is known for the
number of set-valued standard tableaux, cf.\ the discussion
in \cite[$\S$5]{MPY}.  It would be
interesting to show that the number of such tableaux is
\ts {\sf \#P}-complete.


\subsection{}\label{ss:finrem-complexity}
Note that our bijection proving~\eqref{e:Stem-LR} is not computable in
linear time as bijections in~\cite{PV}, nor is it easily comparable with
bijections in~\cite{PV2} since the lengths of parts are not in binary.
It would be interesting to show that in the terminology of~\cite{PV2},
this bijection is linear time equivalent to the bijection in~\cite{HS}.

\subsection{} \label{ss:finrem-Schubert}~\Cref{t:dewitt} has a geometric explanation in terms
of certain \emph{spherical orbit closures} on the type~A flag varieties (see~\cite[Thm~4.58]{HMP3}).
Here, the skew Schur functions are geometric representatives
for \emph{Schubert varieties} indexed by fully commutative permutations,
while the Schur $P$--functions represent certain \emph{involution Schubert varieties}.
The use of \ts $\Red(w^{\delta_n})$ \ts in our proof reflects the fact, shown
in~\cite[Prop.~3.30]{HMP1}, that the varieties in question have the same cohomology
representatives.

From this perspective, the desired bijection follows from applying both \ts $\FC^{-1}$
\ts and \emph{involution Coxeter--Knuth insertion}
\cite[Thm~5.17]{HMP3} to \ts $\Red(w^{\delta_n/\mu})$.
In this setting, the involution Coxeter--Knuth insertion restricts to
the Worley--Sagan insertion, recovering~$\vp$.
We use the parallel theory for fixed-point-free involution Schubert varieties in our
proof of~\Cref{t:K-dewitt} since their $K$-theory is easier to understand.

The corresponding involutions are
$\mathcal{I}$-Grassmannian in the sense of~\cite{HMP3}.
By analogy with the map~$\FC$, it is an interesting open question to
give a direct bijection between reduced involution words for
$\mathcal{I}$-Grassmannian involutions and marked shifted tableaux
of the appropriate shape.  Here, direct can mean either without using
an insertion algorithm or while using $O\bigl(|\lambda|\ts\log |\lambda|\bigr)$ operations.

\subsection{} \label{ss:URT}
Our proof of \Cref{t:K-dewitt} in fact proves a stronger statement.
If $y$ is a FPF-vexillary fixed-point-free involution in the sense of~\cite{HMP4}, we show its symplectic Grothendieck polynomial is a single $GP_\lambda$.
The analogous statement for vexillary permutations appears in \cite[Lem.~5.4]{RTY}.
As a consequence, the insertion tableau $\PSH$ associated to $y$ must be a unique rectification target in the sense of~\cite{BS}.
Identifying unique rectification targets is an interesting and challenging question that has received some attention~\cite{GMPPPRST}


\subsection{} \label{ss:finrem-sim}
The limit curves of random standard Young tableaux are of interest in
integrable probability as in some cases they can be computed exactly.
Most recently, their existence has been shown for a large class of skew
shapes~\cite{Gor,Sun}.  For both the staircase~\cite{AHRV} and the shifted staircase~\cite{LPS},
these limit curves coincide with limit curves for the square~\cite{PR} when restricted to either triangle.

We implemented our bijection for
uniformly sampling from \ts $\SYT(\de_k/b^a)$ in Sage, see the proof of Theorem~\ref{t:DeWitt-random}.  Our code is available \href{https://cocalc.com/share/b8a5580510561b4e7f0950cdbab7c6ab1e73eaac/Mixed%20Shifted%20Insertion.sagews?viewer=share}{online} on CoCalc~\cite{code}.
It would be interesting to see if there are exact formulas for limit shapes
in this case.

\smallskip

\begin{figure}[hbt]
\centering
\begin{tabular}{ccc}
\includegraphics[height=3.cm]{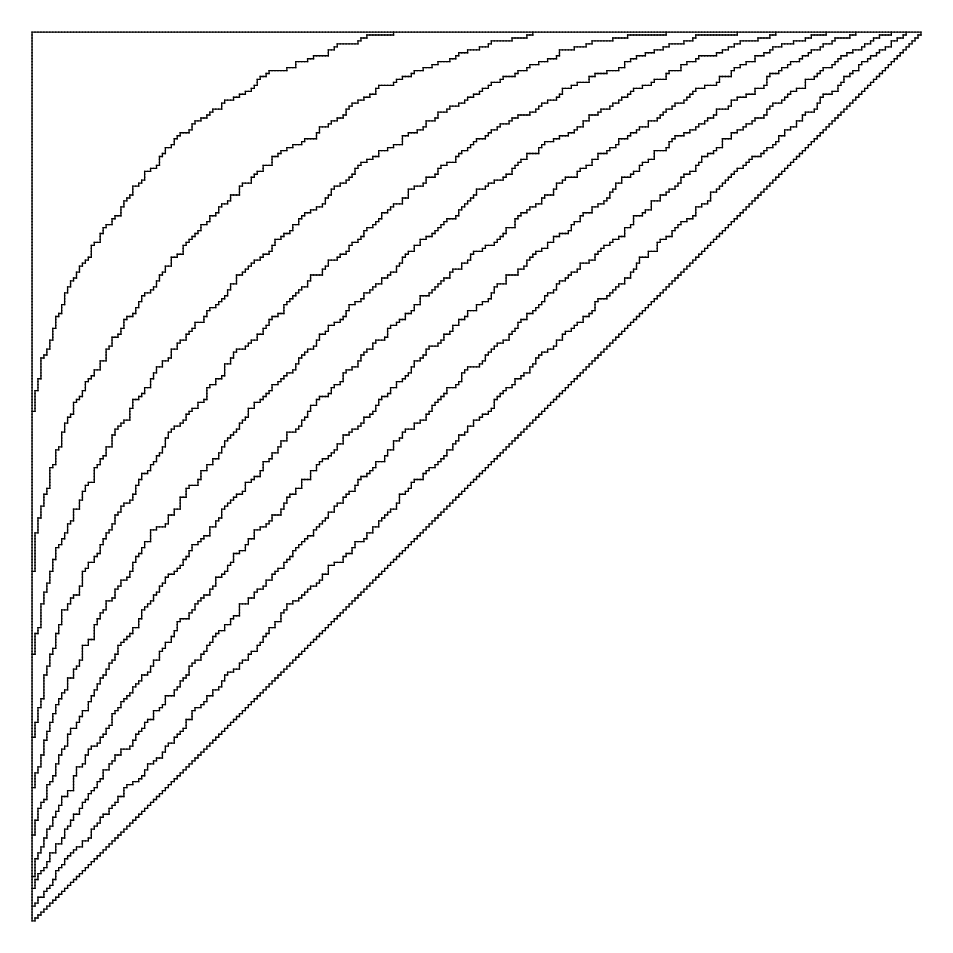} & \hskip1.cm \includegraphics[width=3.cm]{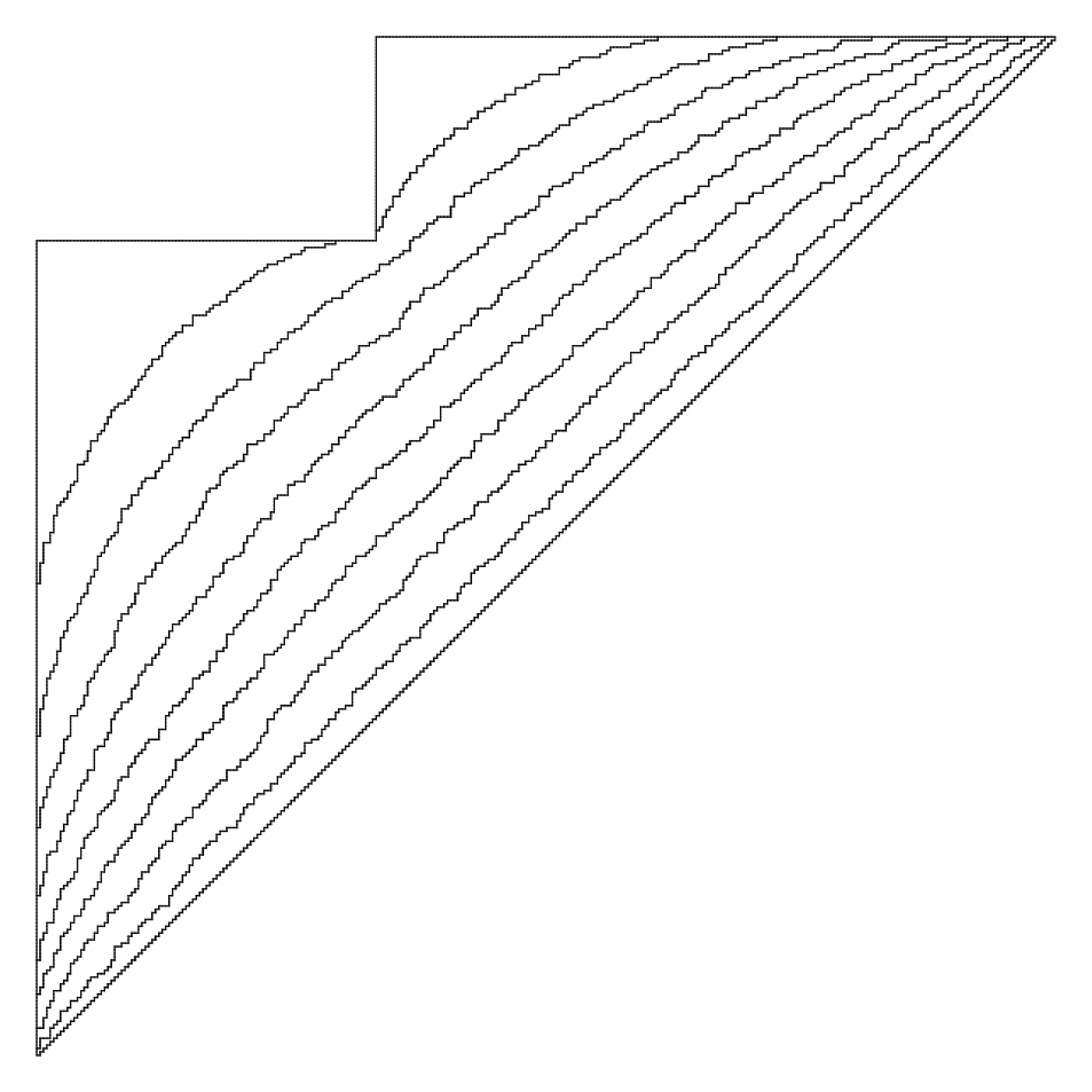} & \hskip1.cm {\includegraphics[width=3.cm]{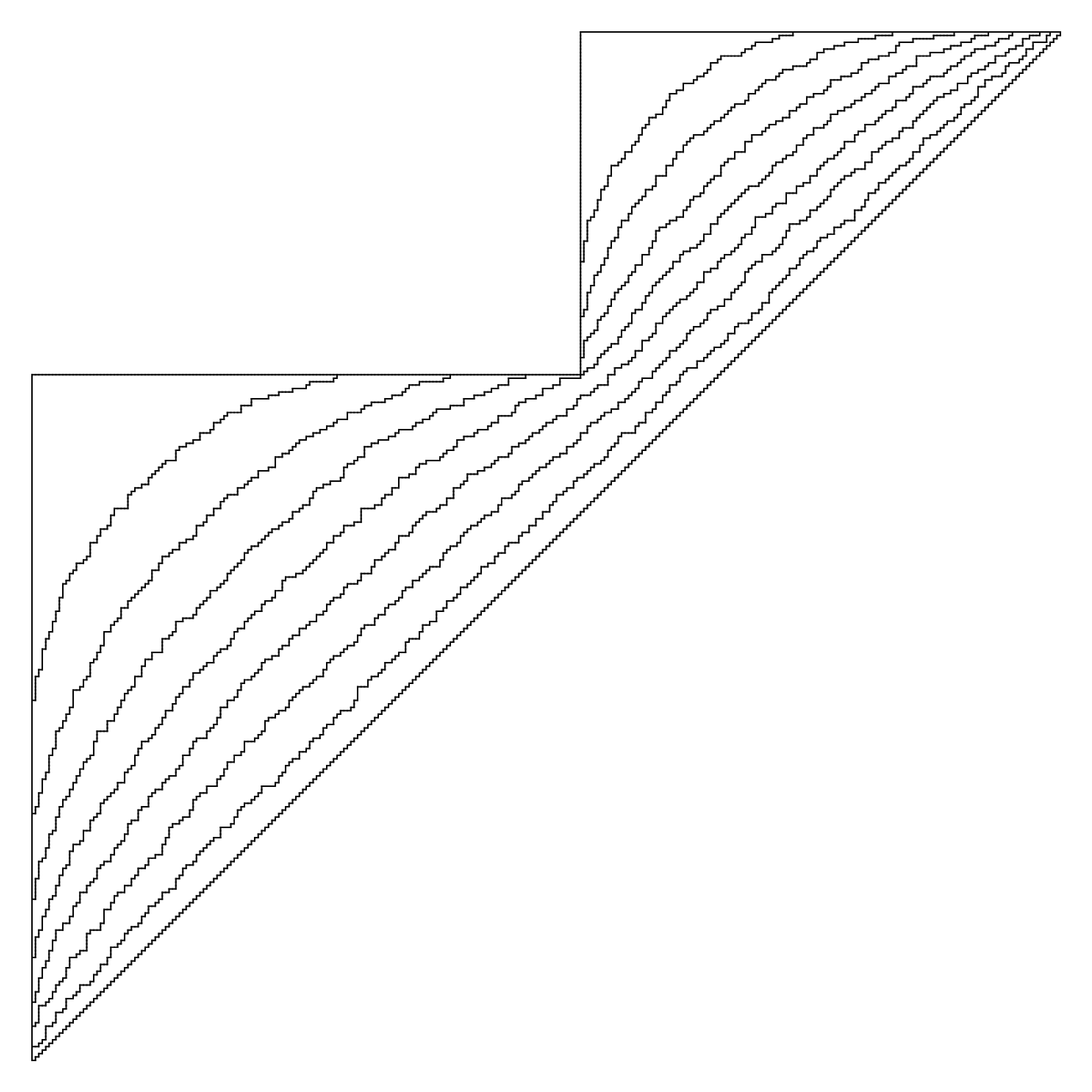}} \\
    & & \\
    \includegraphics[height=3.cm]{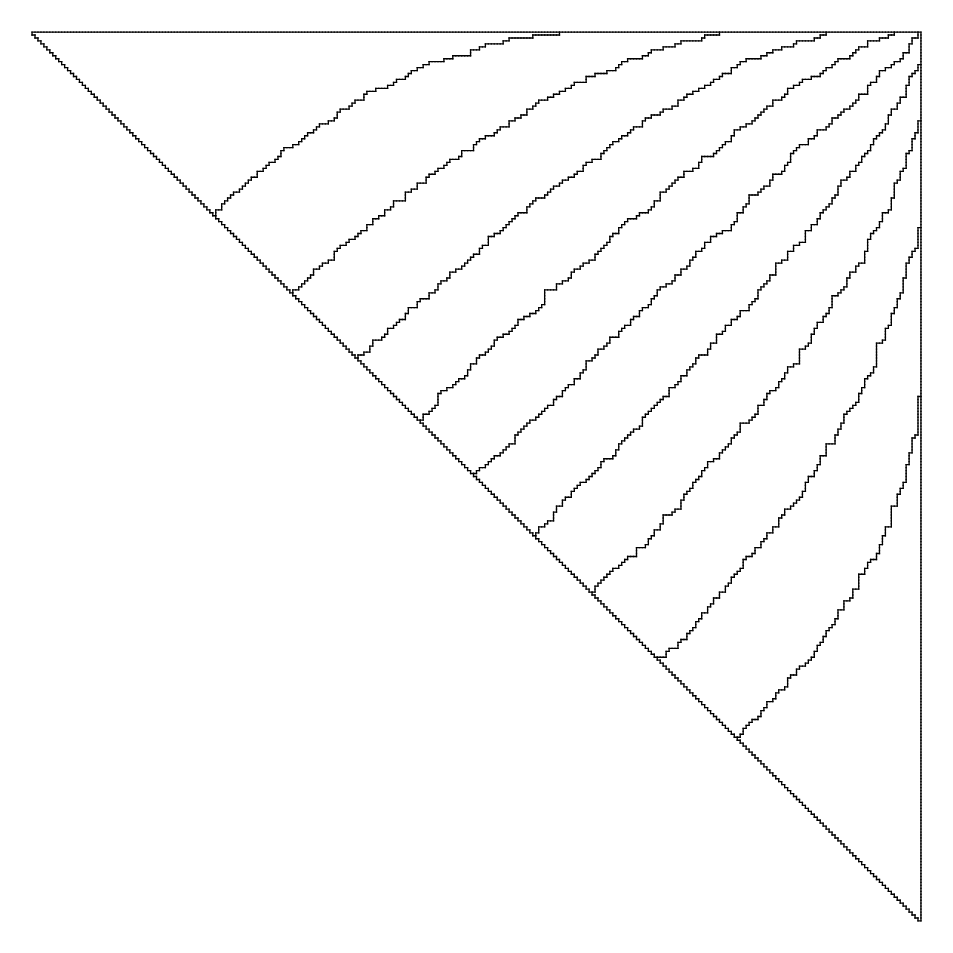} & \hskip1.cm
    {\includegraphics[width=3.cm]{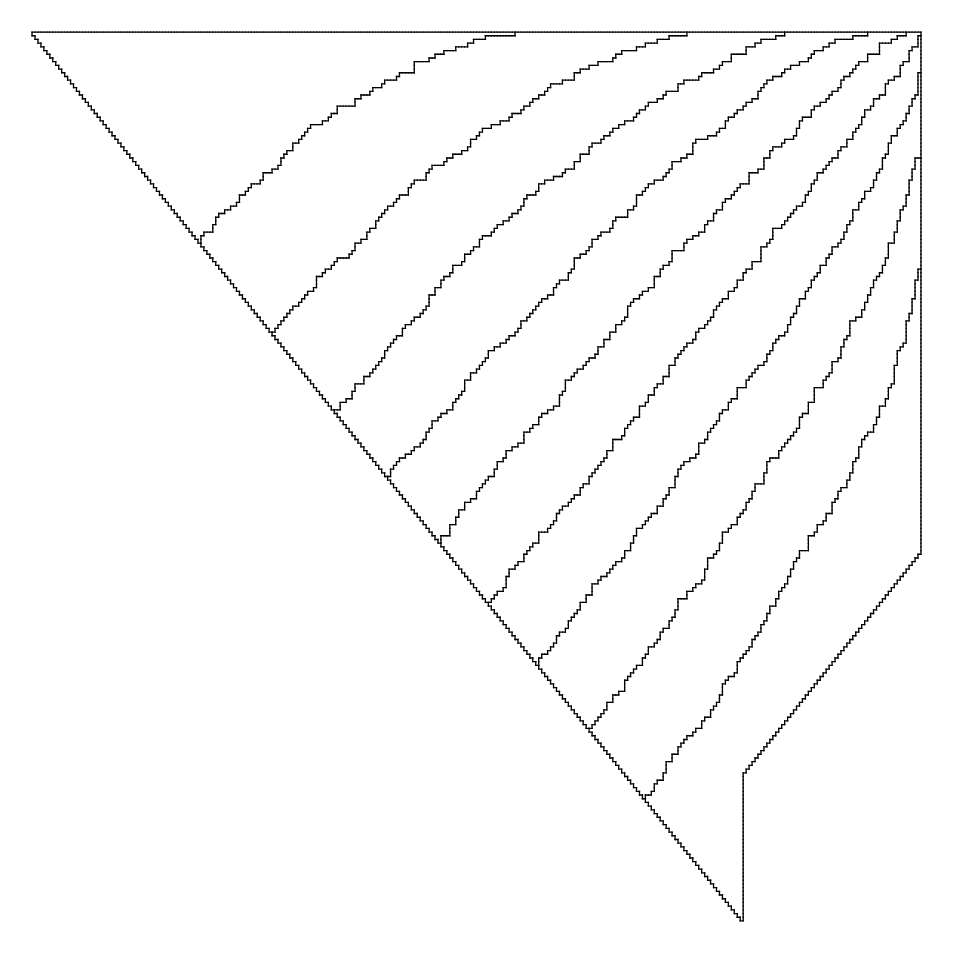}}
        & \hskip1.1cm
    \includegraphics[width=3.2cm]{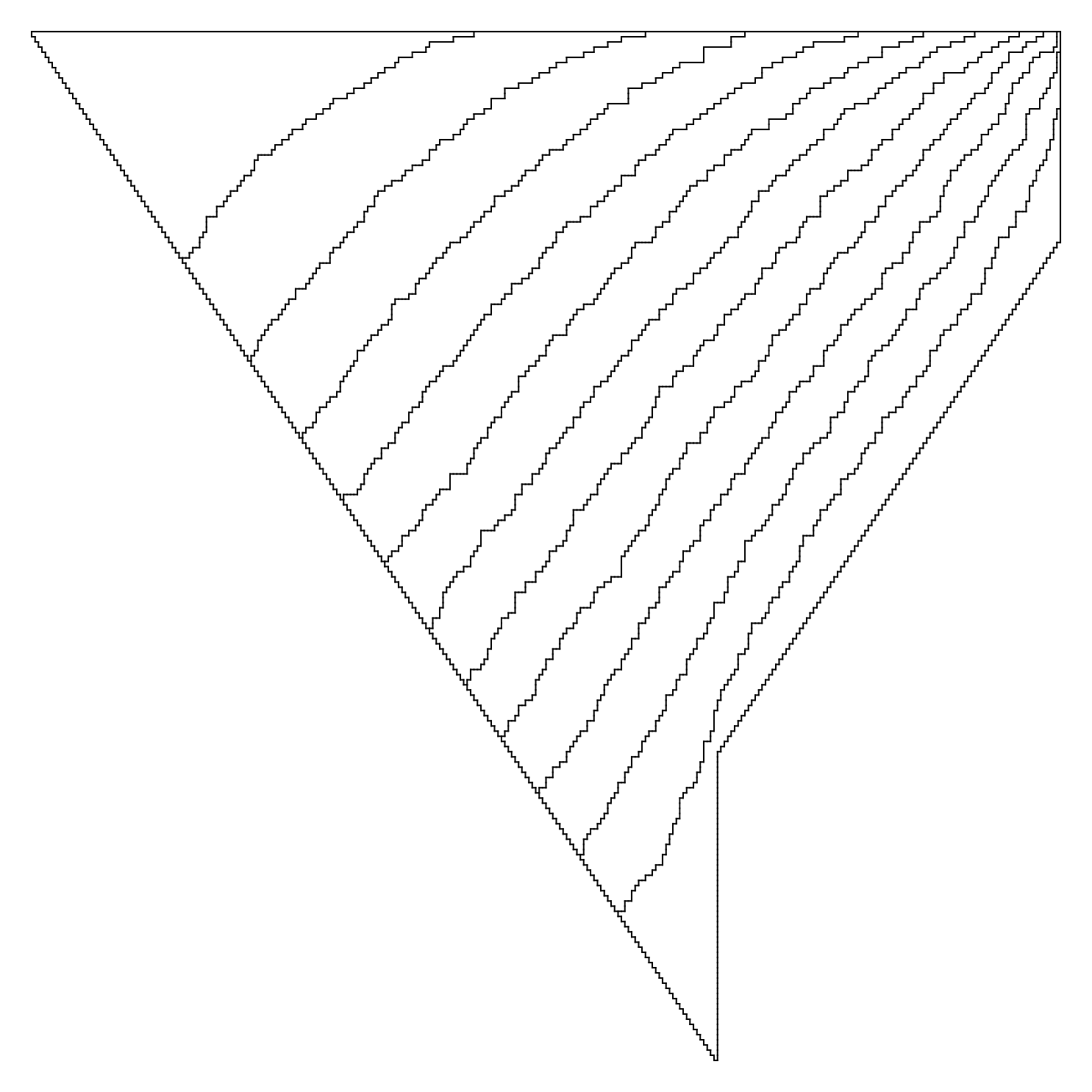}
\end{tabular}
    \caption{Simulations for $k=300$.  \underline{Left column}: \ts $\SYT(\de_k)$ \ts and \ts $\ShSYT(\vr_k)$.
     \underline{Middle column}: \ts $\SYT(\de_k/b^a)$ \ts and \ts $\ShSYT(\vr-\tau^{a,b})$ \ts with \ts  $(a,b)=(60,100)$.
     \underline{Right column}: \ts $\SYT(\de_k/b^a)$ \ts and \ts $\ShSYT(\vr-\tau^{a,b})$ \ts with \ts   $(a,b)=(100,160)$.}
\label{f:sim}
\end{figure}
\smallskip

\subsection{}  \label{ss:finrem-q}
It follows from~\Cref{l:complement} and our bijective proof of~\Cref{t:DeWitt-SYT},
that this theorem has a $q$-analogue where the tableaux are weighted with \ts $q^{\maj(T)}$, where \ts
$\maj(T)$ \ts is the \emph{major index} (see e.g.~\cite[$\S$7.19]{EC}).  It would be
interesting to find a $q$-analogue of~\Cref{t:DeWitt-random}.

Let us note that the same (numerical) $q$-analogue of~\Cref{c:dewitt}
is given in \cite{KS,MPP3}.  Paper \cite{MPP1} gave a combinatorial
interpretation for the GF over semistandard Young tableaux of
skew and shifted skew shapes, which can be viewed as another natural
$q$-analogue of \ts $f^{\la/\mu}=|\SYT(\la/\mu)|$.
Finally, Kerov defined a $q$-hook walk in~\cite{Ker}.  His approach
extends to the \emph{weighted case}~\cite{CKP} and then can be modified
to the \emph{shifted weighted case}~\cite{Kon}, which includes the \emph{$q$-shifted
case} as a special case.

\vskip.5cm

{\small
\subsection*{Acknowledgements}
We are grateful to Darij Grinberg, Eric Marberg, Damir Yeliussizov and Alex Yong for helpful remarks.
Special thanks to Darij Grinberg, Vic Reiner and John Stembridge
for telling us about the history of~\Cref{t:skew}
(see~$\S$\ref{ss:finrem-hist}), and informing us of~\cite{GR} and~\cite{Ste3}.
Additionally, we thank Jianping Pan for pointing out several errors in an earlier draft of this document.
Some of the authors were hanging out at MSRI during the Fall of~2017
when some of this research was done.  We are grateful to MSRI for
the hospitality even if those days filled with wildfire smoke feel
like  memories from another lifetime. The simulations in Figure~\ref{f:sim}
were made using Sage and its algebraic combinatorics features
developed by the Sage-Combinat community~\cite{sage-combinat}.  AHM was partially supported by
the NSF grant DMS-1855536. IP was partially supported by the NSF grants
DMS-1700444 and CCF-2007891.  NW was partially supported by Simons grant 585380.
}


%
{\vskip.9cm }

\end{document}